\documentclass[reqno]{amsart}
\usepackage{amssymb}

\numberwithin{equation}{section}

\newtheorem{theorem}{Theorem}[section]
\newtheorem{proposition}[theorem]{Proposition}

\theoremstyle{definition}

\theoremstyle{definition} 

\newtheorem{remarks}[theorem]{Remarks}

\def\no{|\!|}

\newcommand{\cA}{\mathcal A}
\newcommand{\cB}{\mathcal B}
\newcommand{\cE}{\mathcal E}
\newcommand{\cG}{\mathcal G}

\newcommand{\cM}{\mathcal M}
\newcommand{\R}{\mathbb R}
\newcommand{\C}{\mathbb C}
\newcommand{\E}{\mathbb E}
\newcommand{\F}{\mathbb F}

\newcommand{\N}{\mathbb N}

\newcommand{\st}{\,:\,}
\newcommand{\M}{{\mathcal M}}
\newcommand{\EE}{\mathbb E}

\newcommand{\ra}{\rightarrow}
\def\l{l}

\DeclareMathSymbol{\complement}{\mathord}{AMSa}{"7B}

\def\vv<#1>{\langle #1\rangle}
\def\Vv<#1>{\bigl\langle #1\bigr\rangle}

\begin{document}


\title[Convergence of solutions to equilibria]
{On convergence of solutions to equilibria for quasilinear parabolic problems}

\author[J. Pr\"uss]{Jan Pr\"uss}
\address{Martin-Luther-Universit\"at Halle-Witten\-berg\\
         Institut f\"ur Mathematik \\
         Theodor-Lieser-Strasse 5\\
         D-06120 Halle, Germany}
\email{jan.pruess@mathematik.uni-halle.de}

\author[G. Simonett]{Gieri Simonett}
\address{Department of Mathematics\\
         Vanderbilt University \\
         Nashville, TN~37240, USA}
\email{gieri.simonett@vanderbilt.edu}

\author[R. Zacher]{Rico Zacher}
\address{Martin-Luther-Universit\"at Halle-Witten\-berg\\
         Institut f\"ur Mathematik \\
         Theodor-Lieser-Strasse 5\\
         D-06120 Halle, Germany}
\email{rico.zacher@mathematik.uni-halle.de}

\thanks{The research of G.S. was partially
supported by NSF, Grant DMS-0600870. The research of R.Z. was
partially supported by the Deutsche Forschungsgemeinschaft (DFG),
Bonn, Germany.}

\begin{abstract}
We show convergence of solutions to equilibria for quasilinear
parabolic evolution equations in situations where the set of
equilibria is non-discrete, but forms a finite-dimensional
$C^1$-manifold which is normally hyperbolic. Our results do not
depend on the presence of an appropriate Lyapunov functional as in
the \L ojasiewicz-Simon approach, but are of local nature.
\end{abstract}
\date{\today}
\maketitle
\vspace{-.5cm}

\\
\noindent
\section{Introduction}
The principle of linearized stability is a well-known and powerful tool for proving stability or instability of equilibria of
nonlinear evolution equations. It is known to be true for large classes of nonlinear evolution equations, even for such which
are nonlocal. The literature on this subject is large.
Since here we are mainly interested in quasilinear
parabolic problems, we only refer to the monograph by
Lunardi \cite{Lun95}, and to \cite{Ama90, Pru03}.

In this paper we will consider the following situation:
suppose that for a nonlinear evolution equation we
have a $C^1$-{\em manifold of equilibria} $\cE$ such that at a
point $u_*\in\cE$, the kernel $N(A_0)$ of the linearization $A_0$ is
isomorphic to the tangent space of $\cE$ at $u_*$, the eigenvalue
$0$ of $A_0$ is semi-simple, and the remaining spectral part of the
linearization $A_0$ is stable. Then solutions starting nearby $u_*$
exist globally and converge to some point on $\cE$.
This result is well-known to specialists 
in the area of dynamical systems (where it is considered a folk theorem), but might be less familiar to people in the PDE community.

The situation described above occurs frequently in
applications. We call it the {\em generalized principle of
linearized stability}, and the equilibrium $u_*$ is then termed
{\em normally stable}.

A typical example for this situation to occur is the
case where the equations under consideration involve symmetries,
i.e.\ are invariant under the action of a Lie-group $\cG$. If then
$u_*$ is an equilibrium, the manifold $\cE$ includes the action of
$\cG$ on $u_*$ and the manifold $\cG u_*$ is a subset of $\cE$.

A standard method to handle situations as described above is to
refer to {\em center manifold theory}. In fact in that situation
the center manifold of the problem in question will be unique, and
it coincides with $\cE$ near $u_*$. Thus the so-called {\em
shadowing lemma} in center manifold theory implies the result.
Center manifolds are well-studied objects in the theory of nonlinear
evolution equations. For the parabolic case we refer to the
monographs \cite{Hen81,Lun95}, and to 
the publications 
\cite{BaJo89, BJL00, DPLu88,LPS06,LPS08,Mie91,Sim95, VI92}.

However, the theory of center manifolds 
is a technically difficult matter. It usually involves higher regularity of
the involved nonlinearities -
in particular concerning the shadowing property. 
Therefore it seems desirable to have a simpler, direct approach to the generalized principle of linearized stability which avoids the technicalities of center manifold theory.

The purpose of this paper is to present such an approach. It turns
out that the effort is only slightly larger than that for the proof
of the standard linearized stability result - which is simple. 
We emphasize that our approach requires only $C^1$-regularity for the nonlinearities.
By several examples we will illustrate 
that our result is applicable to a variety of 
interesting problems in different areas of applied
analysis. It is our belief that the approach devised in this
manuscript will be fruitful for the stability analysis of
equilibria for parabolic evolution equations that involve symmetries
in the way described above.

Here we would also like to mention the work
in  \cite{Cui07}, where the action of a Lie group 
has been used for the stability analysis of equilibrium solutions.
However, the approach given here is considerably more general and
flexible.

In Section 2 we formulate and prove our main result for abstract
autonomous quasilinear parabolic problems. Theorem~\ref{th:1}
implies, for instance, the main result in \cite{EsSi98} on
convergence of solutions for the Mullins-Sekerka problem. We also
show by means of examples that the conditions of Theorem~\ref{th:1}
are  necessary in order to have convergence to a single equilibrium.

In Section 3, we consider quasilinear parabolic systems with
nonlinear boundary conditions and we show that our techniques can
also be applied to this situation. Sections 4 and 5 illustrate the
scope of our main result, as we show convergence towards
equilibria for the Mullins-Sekerka model, and stability of
travelling waves for a quasilinear parabolic equation.

In Section 6 we consider the so-called {\em normally hyperbolic}
case, where the remaining part of the spectrum of $A_0$ also
contains an unstable part away from the imaginary axis. In this
situation, one cannot expect convergence of all solutions starting
near $u_*$, but only for those initial values which are on the
stable manifold.

To cover the quasilinear case our approach makes use of maximal $L_p$-regularity in an essential way. As general references
for this theory we refer to the recent publications
 \cite{DHP03, DHP06}, to the survey article \cite{Pru03},
and also to \cite{Ama95, Am04, Am05, ClLi94, Lun95}.

In a forthcoming paper these results are extended to the case
where the boundary conditions are of relaxation type, i.e.\ are coupled
with an evolution equation on the boundary, as in \cite{DPZ06}.
Problems of the last kind are important e.g.\ for the Stefan problem with
surface tension, see\ \cite{EPS03,PrSi06},
and for the two-phase Navier-Stokes problem with a free boundary.

Finally, we should like to point out that the generalized
principle of linearized stability described in the current paper
can also be adapted and applied to fully nonlinear parabolic 
equations, see \cite{PrSiZa08}.
\noindent
\section{Convergence for abstract quasilinear problems}

Let $X_0$ and $X_1$ be two Banach spaces such that
$X_1\hookrightarrow X_0$, i.e. $X_1$ is continuously and densely
embedded in $X_0$. In this section we consider the autonomous
quasilinear problem
\begin{equation}
\label{u-equation}
\dot{u}(t)+A(u(t))u(t)=F(u(t)),\quad t>0, \quad u(0)=u_0.
\end{equation}
For $1<p<\infty$ we introduce the real interpolation space
$X_\gamma:=(X_0,X_1)_{1-1/p,p}$ and we assume that there is an open
set $V\subset X_\gamma$ such that
\begin{equation}
\label{AF}
(A,F)\in C^1(V,\cB(X_1,X_0)\times X_0).
\end{equation}
Here $\cB(X_1,X_0)$ denotes the space of all bounded
linear operators from $X_1$ into $X_0$.
In the sequel we use the notation $|\cdot|_j$ to denote the norm
in the respective spaces $X_j$ for $j=0,1,\gamma$.
Moreover, for any normed space $X$,
$B_X(u,r)$ denotes the open ball in $X$ with radius $r>0$ around $u\in X$.
\smallskip\\
\noindent
Let $ \cE\subset V\cap X_1$ denote the set of  equilibrium solutions of (\ref{u-equation}), which means that
$$
u\in\cE \quad \mbox{ if and only if }\quad u\in V\cap X_1,
\; A(u)u=F(u).
$$
Given an  element $u_*\in\cE$,  we assume that $u_*$ is
contained in an $m$-dimensional manifold of equilibria. This means that there
is an open subset $U\subset\R^m$, $0\in U$, and a $C^1$-function
$\Psi:U\rightarrow X_1$,  such that
\begin{equation}
\label{manifold}
\begin{aligned}
& \bullet\
\text{$\Psi(U)\subset \cE$ and $\Psi(0)=u_*$,} \\
& \bullet\
 \text{the rank of $\Psi^\prime(0)$ equals $m$, and} \\
& \bullet\
\text{$A(\Psi(\zeta))\Psi(\zeta)=F(\Psi(\zeta)),\quad \zeta\in U.$}
\end{aligned}
\end{equation}
We assume further that near $u_*$ there are no other equilibria
than those given by $\Psi(U)$,
i.e.\ $\cE\cap B_{X_1}(u_*,{r_1})=\Psi(U)$, for some $r_1>0$.
\smallskip\\
\noindent
We suppose that
the operator $A(u_*)$ has the property of maximal $L_p$-regularity.
Introducing  the deviation $v=u-u_*$ from the equilibrium $u_*$, the equation
for $v$ then reads as
\begin{equation}
\label{v-equation}
\dot{v}(t)+A_0v(t)=G(v(t)),\quad t>0,\quad v(0)=v_0,
\end{equation}
where
$v_0=u_0-u_*$
and
\begin{equation}
\label{A0}
A_0v=A(u_*)v+(A^\prime(u_*)v)u_*-F^\prime(u_*)v\quad\hbox{for }
v\in X_1.
\end{equation}
The function $G$ can be written as $G(v)=G_1(v)+G_2(v,v)$, where
\begin{equation*}
\begin{aligned}
G_1(v)&=(F(u_*+v)\!-\!F(u_*)\!-\!F^\prime(u_*)v)
\!-\!(A(u_*+v)\!-\!A(u_*)\!-\!A^\prime(u_*)v)u_*, \\
G_2(v,w)&=-(A(u_*+v)\!-\!A(u_*))w,\quad  w\in X_1,\;v\in V_\ast,
\end{aligned}
\end{equation*}
where $V_\ast:=V-u_\ast$. It follows from \eqref{AF}
that
$G_1\in C^1(V_\ast,X_0)$ and also that
$G_2\in C^1(V_\ast\times X_1,X_0)$.
Moreover, we have
\begin{equation}
\label{G(0)}
G_1(0)=G_2(0,0)=0,
\quad G_1^\prime(0)=G_2^\prime(0,0)=0,
\end{equation}
where $G^\prime_1$ and $G^\prime_2$ denote the Fr\'{e}chet
derivatives of $G_1$ and $G_2$, respectively.
\medskip\\
Setting
$\psi(\zeta)=\Psi(\zeta)-u_*$ results in the following equilibrium
equation for problem \eqref{v-equation}
\begin{equation}
\label{equilibrium-psi}
A_0\psi(\zeta)=G(\psi(\zeta)),\quad \mbox{ for all }\;\zeta\in U.
\end{equation}
Taking the derivative with respect to $\zeta$ and using
the fact that $G^\prime(0)=0$ we conclude that
$A_0\psi^\prime(0)=0$ and this implies that
\begin{equation}
\label{tangent-space}
T_{u_\ast}(\cE)\subset N(A_0),
\end{equation}
where $T_{u_\ast}(\cE) $ denotes the tangent space of $\cE$ at
$u_\ast$.
\medskip\\
\noindent
After these preparations we can state the following result on convergence of  solutions starting near $u_*$.
\goodbreak
\begin{theorem}
\label{th:1} Let $1<p<\infty$. Suppose $u_*\in V\cap X_1$ is an
equilibrium of (\ref{u-equation}), and suppose that the functions
$(A,F)$ satisfy \eqref{AF}. Suppose further that $A(u_*)$ has the
property of maximal $L_p$-regularity. Let $A_0$, defined in
\eqref{A0}, denote the linearization of (\ref{u-equation}) at $u_*$.
Suppose that $u_*$ is normally stable, i.e.\ assume that
\begin{itemize}
\item[(i)] near $u_*$ the set of equilibria $\cE$ is a $C^1$-manifold in $X_1$ of dimension $m\in\N$,
\vspace{-4mm}
\item[(ii)] \, the tangent space for $\cE$ at $u_*$ is given by $N(A_0)$,
\item[(iii)] \, $0$ is a semi-simple eigenvalue of
$A_0$, i.e.\ $ N(A_0)\oplus R(A_0)=X_0$,
\item[(iv)] \, $\sigma(A_0)\setminus\{0\}\subset \C_+=\{z\in\C:\, {\rm Re}\, z>0\}$.
\end{itemize}
Then $u_*$ is stable in $X_\gamma$, and there exists $\delta>0$ such
that the unique solution $u(t)$ of (\ref{u-equation}) with initial
value $u_0\in X_\gamma$ satisfying $|u_0-u_*|_{\gamma}<\delta$
exists on $\R_+$ and converges at an exponential rate in $X_\gamma$
to some $u_\infty\in\cE$ as $t\rightarrow\infty$.
\end{theorem}
\begin{proof}
(a) Note first that assumption (iii) implies that $0$ is an isolated
spectral point of $\sigma(A_0)$, the spectrum of $A_0$. According to
assumption (iv) $\sigma(A_0)$ admits a decomposition into two
disjoint nontrivial parts with
\begin{equation*}
\sigma(A_0)=\{0\}\cup \sigma_s , \quad
\sigma_s\subset \C_+=\{z\in\C:\,{\rm Re}\, z>0\}.
\end{equation*}
The spectral set $\sigma_c:=\{0\}$ corresponds to the center part, and
$\sigma_s$ to the stable part of the analytic $C_0$-semigroup $e^{-A_0 t}$, or equivalently
of the Cauchy problem $\dot{w}+A_0w=f$.
\smallskip\\
\noindent In the following, we let $P^l$, $l\in\{c,s\}$, denote the
spectral projections according to the spectral sets $\sigma_c=\{0\}$
and $\sigma_s$, and we set $X^l_j:=P^l X_j$ for $l\in\{c,s\}$ and
$j\in\{0,1,\gamma\}$. The spaces $X^\l_j$ are equipped with the
norms $|\cdot| _j$ for $j=0,1,\gamma$. We have the topological
direct decomposition
\begin{equation*}
X_1=X_1^c\oplus X_1^s,\quad X_0=X_0^c\oplus X_0^s,
\end{equation*}
and this decomposition reduces $A_0$ into $A_0=A_c\oplus A_s$,
where  $A_l$  is the part of $A_0$ in $X^l_0$
for $l\in\{c,s\}$.
Since $\sigma_c=\{0\}$ is compact it follows that
$X_0^c\subset X_1$.
Therefore, $X_0^c$ and $X_1^c$ coincide as
vector spaces. In the following, we will
just write $X^c=(X^c,|\cdot|_j)$
for either of the spaces $X^c_0$ and $X^c_1$.
We note that $N(A_0)$, the kernel of $A_0$, is contained in $X^c$.
The operator $A_s$
inherits the property of $L_p$-maximal regularity from $A_0$.
Since $\sigma(A_s)=\sigma_s\subset \C_+$
we obtain that the Cauchy problem
\begin{equation}
\label{Cauchy}
\dot w + A_s w=f,\quad w(0)=0,
\end{equation}
also enjoys the property of maximal regularity,
even on the interval $J=(0,\infty)$.
In fact the following estimates are true.
For any $a\in (0,\infty]$
let
\begin{equation}
\label{EE}
\EE_0(a)=L_p((0,a);X_0),
\quad \EE_1(a)=H^1_p((0,a);X_0)\cap L_p((0,a);X_1).
\end{equation}
The natural norms in $\EE_j(a)$ will be denoted by 
$|\!|\cdot|\!|_{\EE_j(a)}$ for $j=0,1$. Then the  Cauchy problem \eqref{Cauchy} has for
each
 $f\in L_p((0,a);X_0^s)$ a unique solution
$$
w\in H^1_p((0,a);X^s_0)\cap L_p((0,a);X^s_1),
$$
and there exists a constant $M_0$ such that $|\!|w
|\!|_{\EE_1(a)}\le M_0 |\!|f |\!|_{\EE_0(a)}$ for every $a>0$, and
every function $f\in L_p((0,a);X^s_0)$. In fact, since
$\sigma(A_s-\omega)$ is still contained in $\C_+$ for $\omega$ small
enough, we see that the operator $A_s-\omega$ enjoys the same
properties as $A_s$. Therefore, every solution of the Cauchy problem
$\eqref{Cauchy}$ satisfies the estimate
\begin{equation}
\label{M0}
|\!|e^{\sigma t} w |\!|_{\EE_1(a)}
\le M_0 |\!|e^{\sigma t}f |\!|_{\EE_0(a)},
\quad \sigma\in [0,\omega], \quad a>0,
\end{equation}
for $f\in L_p((0,a);X^s_0)$, where $M_0=M_0(\omega)$ for $\omega>0$
fixed; cf.\ \cite[Sec.\ 6]{Pru03}. Furthermore, there exists a constant
$M_1>0$ such that
\begin{equation}
\label{M1}
|\!|e^{\omega t}e^{-A_s t}P^su|\!|_{\EE_1(a)}
+\sup_{t\in [0,a)}|e^{\omega t}e^{-A_st}P^su|_\gamma
\le M_1|P^su|_\gamma\,
\end{equation}
for every $u\in X_\gamma$ and $a\in (0,\infty]$.
For future use we note that \begin{equation}
\label{zero-trace}
\sup_{t\in [0,a)}|w(t)|_\gamma\le c_0
\no w\no_{\EE_1(a)}
\quad\text{for all $w\in \EE_1(a)$ with $w(0)=0$}
\end{equation}
with a constant $c_0$ that is independent of $a\in (0,\infty]$,
see for instance the proof of \cite[Proposition 6.2]{PSS07}.
We remind that $N(A_0)$ is contained in $X^c$.
\medskip\\
\noindent
(b) It follows from the considerations above and
assumptions (i)-(iii) that in fact
\begin{equation*}
N(A_0)=X^c\quad\text{and}\quad \text{dim}(X^c)=m.
\end{equation*}
As $X^c$ has finite dimension,
the norms $|\cdot|_j$ for $j=0,1,\gamma$ are equivalent,
and we equip $X^c$ with one of these equivalent norms,
say with $|\cdot|_0$.
Let us now consider the mapping
\begin{equation*}
g:U\subset \R^m\to X^c,\quad
g(\zeta):=P^c\psi(\zeta),\quad \zeta\in U.
\end{equation*}
It follows from our assumptions that
$g^\prime(0)=P^c\psi^\prime(0):\R^m\to X^c$
is an isomorphism
(between the finite dimensional spaces $\R^m$ and $X^c$).
By the inverse function theorem, $g$ is a
$C^1$-diffeomorphism of a neighborhood of $0$ in $\R^m$
into a neighborhood, say $B_{X^c}(0,\rho_0)$, of $0$ in $X^c$.
Let $g^{-1}:B_{X^c}(0,\rho_0)\to U$ be the inverse mapping.
Then $g^{-1}:B_{X^c}(0,\rho_0)\to U$ is $C^1$ and $g^{-1}(0)=0$.
Next we set
$\Phi(x):=\psi(g^{-1}(x))$ for $x\in B_{X^c}(0,\rho_0)$
and we note that
\begin{equation*}
\Phi\in C^1(B_{X^c}(0,\rho_0),X_1^s), \quad
\Phi(0)=0,
\quad \{\Phi(x)+u_\ast \st x\in B_{X^c}(0,\rho_0)\}=\cE\cap W,
\end{equation*}
where $W$ is an appropriate neighborhood of $u_\ast$ in $X_1$.
One readily verifies that
\begin{equation*}
P^c \Phi(x)=((P^c\circ \psi)\circ g^{-1})(x)=
(g\circ g^{-1})(x)=x,\quad x\in  B_{X^c}(0,\rho_0),
\end{equation*}
and this yields
$\Phi(x)=P^c\Phi(x)+P^s\Phi(x)=x+P^s\Phi(x)$ for
$x\in B_{X^c}(0,\rho_0)$.
Setting $\phi(x):=P^s\Phi(x)$ we conclude that
\begin{equation}
\label{phi}
\phi\in C^1(B_{X^c}(0,\rho_0),X_1^s),\quad \phi(0)=\phi^\prime (0)=0,
\end{equation}
and that
$$
\{x+\phi(x)+u_\ast \st x\in B_{X^c}(0,\rho_0)\}=\cE\cap W,
$$
where $W$ is a neighborhood of $u_\ast$ in $X_1$.
This shows that the manifold $\cE$
can be represented
as the (translated) graph of the function $\phi$ in a neighborhood
of $u_\ast$. Moreover,
the tangent space of $\cE$ at $u_\ast$
coincides with $N(A_0)=X^c$.
By applying the projections $P^l$, $l\in\{c,s\}$, to equation \eqref{equilibrium-psi}
and using that $x+\phi(x)=\psi(g^{-1}(x))$
for $x\in B_{X^c}(0,\rho_0)$, and that $A_c\equiv 0$,
we obtain the following equivalent system of equations
for the equilibria of \eqref{v-equation}
\begin{equation}
\label{equilibria-phi}
P^cG(x+\phi(x))=0,\quad
P^s G(x+\phi(x))=A_s\phi(x),
\quad x\in B_{X_c}(0,\rho_0).
\end{equation}
Finally, let us also agree that $\rho_0$ has already been chosen small enough
so that
\begin{equation}
\label{estimate-phi}
|\phi^\prime(x)|_{\cB(X^c,X_1^s)}\le 1 ,\quad
|\phi(x)|_1\le  |x|,\quad x\in B_{X^c}(0,\rho_0).
\end{equation}
This can always be achieved, thanks to \eqref{phi}.
\medskip\\
\noindent
(c) Introducing the new variables
\begin{equation*}
\begin{aligned}
&x=P^c v=P^c (u-u_*), \\
&y=P^sv-\phi(P^cv)=P^s(u-u_*)-\phi(P^c (u-u_*))
\end{aligned}
\end{equation*}
we then
obtain the following system of evolution equations
in $X^c\times X^s_0 $
\begin{equation}
\label{system}
\left\{
\begin{aligned}
\dot{x}=T(x,y),      \quad &x(0)=x_0, \\
\dot{y}+A_sy=R(x,y), \quad &y(0)=y_0,\\
\end{aligned}
\right.
\end{equation}
with $x_0=P^cv_0$ and $y_0=P^sv_0-\phi(P^cv_0)$,
where the functions $T$ and $R$ are given by
\begin{equation*}
\begin{aligned}
&T(x,y)=P^c G(x+\phi(x)+y), \\
&R(x,y)=P^sG(x+\phi(x)+y)-A_s\phi(x)-\phi^\prime(x)T(x,y).
\end{aligned}
\end{equation*}
Using the equilibrium equations \eqref{equilibria-phi}, the
expressions for $R$ and $T$ can be rewritten as
\begin{equation}
\label{R-T}
\begin{aligned}
&T(x,y)=P^c \big(G(x+\phi(x)+y)-G(x+\phi(x))\big), \\
&R(x,y)=P^s \big(G(x+\phi(x)+y)-G(x+\phi(x))\big)-\phi^\prime(x)T(x,y).
\end{aligned}
\end{equation}
Although the term $P^cG(x+\phi(x))$ in $T$ is zero, see
\eqref{equilibria-phi}, we include it here for reasons of symmetry,
and for justifying the estimates for $T$ below. Equation \eqref{R-T}
immediately yields
\begin{equation*}
\label{R=T=0}
T(x,0)=R(x,0)=0\quad\text{for all }\ x\in B_{X^c}(0,\rho_0),
\end{equation*}
showing that
the equilibrium set $\cE$ of \eqref{u-equation}
near $u_*$ has been reduced to the set
$ B_{X^c}(0,\rho_0)\times \{0\}\subset X^c\times X^s_1$.
\medskip\\
Observe also that there is a unique correspondence between the solutions of \eqref{u-equation}
close  to $u_*$ in $X_\gamma$ and those of (\ref{system}) close to $0$.
We call  system \eqref{system} the {\em normal form}
of \eqref{u-equation} near
its normally stable equilibrium $u_*$.
\medskip\\
(d) From the representation of $G$ and \eqref{G(0)}
we  obtain the following estimates for $G_1$ and $G_2$:
for given $\eta>0$ we may choose $r=r(\eta)>0$ small enough such that
$$
|G_1(v_1)-G_1(v_2)|_{0}\le \eta |v_1-v_2|_\gamma,
\quad  v_1,v_2\in B_{X_\gamma}(0,r).
$$
Moreover, there is a constant $L>0$
such that
\begin{equation*}
\begin{aligned}
&|G_2(v_1,w)-G_2(v_2,w)|_0
\le L |w|_{1}\,|v_1-v_2|_\gamma,
&& w\in X_1, &&v_1,v_2\in B_{X_\gamma}(0,r), \\
&|G_2(v,w_1)-G_2(v,w_2)|_0 \le L\; r\;|w_1-w_2|_1,
&& w_1,w_2\in X_1, &&v\in B_{X_\gamma}(0,r).
\end{aligned}
\end{equation*}
We remark that $L$
does not depend on $r\in (0,r_0]$
with $r_0$ appropriately chosen.
Combining these estimates we have
\begin{equation}
\label{G-estimate}
\begin{aligned}
|G(v_1)-G(v_2)|_0 &\le
\big(\eta+L|v_2|_1\big)|v_1-v_2|_\gamma +Lr|v_1-v_2|_1 \\
&\le
C_0\big(\eta+r+|v_2|_1\big)|v_1-v_2|_1
\end{aligned}
\end{equation}
for all $v_1,v_2\in  B_{X_\gamma}(0,r)\cap X_1$,
where  $C_0$ is independent of $r\in (0,r_0]$.
\medskip\\
In the following, we will always assume
that $r\in (0,r_0]$ and $r_0\le 3\rho_0$.
Taking $v_1=x+\phi(x)+y$ and $v_2=x+\phi(x)$
in \eqref{G-estimate}
we infer from
\eqref{estimate-phi} and \eqref{R-T} that
\begin{equation}
\label{estimate-R-T}
\begin{aligned}
|T(x,y)|,\ |R(x,y)|_0 \le C_1\big(\eta+r+|x+\phi(x))|_1\big)|y|_1
\le \beta |y|_1,
\end{aligned}
\end{equation}
for all $x\in \bar B_{X^c}(0,\rho)$,
$y\in \bar B_{X^s_\gamma}(0,\rho)\cap X_1$
and all $\rho\in (0,r/3)$,
where $\beta=C_2(\eta+r)$, and
where $C_1$ and $C_2$ are uniform constants.
Suppose that $\eta$ and, accordingly, $r$ were already chosen
small enough so that
\begin{equation}
\label{beta}
M_0\beta= M_0C_2(\eta +r)\le 1/2.
\end{equation}
\\
(e) By \cite[Theorem 3.1]{Pru03}, problem (\ref{v-equation}) admits
for each $v_0\in B_{X_\gamma}(0,r)$ a unique local strong solution
\begin{equation}
\label{solution}
v\in \E_1(a)\cap C([0,a]; X_\gamma)
\end{equation}
for some number $a>0$.
This solution can be extended to a maximal interval of existence $[0,t_*)$.
If $t_*$ is finite, then either $v(t)$ leaves the ball
$B_{X_\gamma}(0,r)$ at time $t_*$, or the limit $\lim_{t\ra t_*} v(t)$
does not exist in $X_\gamma$.
We show that this cannot happen for initial values
$v_0\in B_{X_\gamma}(0,\delta)$, with $\delta\le r$
to be chosen later.
\medskip\\
\noindent Suppose that $x_0\in B_{X^c}(0,N\delta)$ and $y_0\in
B_{X^s_\gamma}(0,N\delta)$ are given, where the number $\delta$ will
be determined later and $N:=\no P^c\no_{\cB(X_0)}+\no
P^s\no_{\cB(X_\gamma)}$. Let $t_\ast$ denote the existence time for
the solution $(x(t),y(t))$ of system \eqref{system} with initial
values $(x_0,y_0)$, or equivalently, for the solution $v(t)$ of
\eqref{v-equation} with initial value $v_0=x_0+\phi(x_0)+y_0$. Let
$\rho$ be fixed so that the estimates in \eqref{estimate-R-T} hold.
Set
\begin{equation*}
t_1:=t_1(x_0,y_0):=\sup\{t\in (0,t_\ast)\st
|x(\tau)|,\,|y(\tau)|_\gamma\le \rho,\ \tau\in [0,t]\}
\end{equation*}
and suppose that $t_1<t_\ast$.
Due to \eqref{M0}--\eqref{M1} and \eqref{estimate-R-T}
we obtain
\begin{equation*}
\begin{aligned}
 \no e^{\omega t}y|\!|_{\EE_1(t_1)}
& \le M_1|y_0|_\gamma + M_0|\!|e^{\omega t}R(x,y)|\!|_{\EE_0(t_1)} \\
&\le M_1|y_0|_\gamma +M_0\beta |\!|e^{\omega t}y|\!|_{\EE_1(t_1)}.
\end{aligned}
\end{equation*}
This yields with \eqref{beta}
\begin{equation}
\label{y-estimate-1}
\no e^{\sigma t}y|\!|_{\EE_1(t_1)}\le 2 M_1|y_0|_\gamma\,,
\quad\sigma\in [0,\omega].
\end{equation}
Using this estimate as well as
\eqref{M1}--\eqref{zero-trace} we further have for $t\in [0,t_1)$
\begin{equation*}
\begin{aligned}
|e^{\sigma t}y(t)|_{\gamma}&\le |e^{\sigma t}y(t)-e^{\sigma
t}e^{-A_s t}y_0|_{\gamma}+|e^{\sigma t}e^{-A_s t}y_0|_{\gamma}\\
&\le c_0
\no e^{\sigma t}y(t)-e^{\sigma t}e^{-A_st}y_0\no_{\EE_1(t_1)}+M_1|y_0|_{\gamma}\\
& \le (3c_0M_1+M_1)|y_0|_{\gamma},
\end{aligned}
\end{equation*}
which yields with $M_2=(3c_0+1) M_1$,
\begin{equation}
\label{y-estimate-gamma}
|y(t)|_{\gamma}\le M_2e^{-\sigma t}|y_0|_{\gamma},\quad
t\in [0,t_1),\ \sigma\in [0,\omega].
\end{equation}
We deduce  from the equation for $x$,
the estimate for $T$ in \eqref{estimate-R-T},
and H\"older's inequality that
\begin{equation*}
\begin{aligned}
|x(t)|\,& \le |x_0|+\int_0^t|T(x(s),y(s))|\,ds \\
&\le |x_0|+\beta\int_0^t |y(s)|_1\,ds\\
&\le |x_0|+\beta\Big(\int_0^\infty e^{-\omega
sp^\prime}\,ds\Big)^{1/p\prime} \,
\no e^{\omega t}y\no_{\EE_1(t_1)} \\
& = |x_0|+ \beta c_1\no e^{\omega t}y\no_{\EE_1(t_1)} \le
|x_0|+M_3|y_0|_{\gamma},\quad t\in [0,t_1),
\end{aligned}
\end{equation*}
where $M_3=M_1 c_1/M_0$ and $c_1=(1/[\omega
p^\prime])^{1/p^\prime}$.
Summarizing, we have shown that
$|x(t)|+|y(t)|_\gamma\le |x_0|+(M_2+M_3)|y_0|_\gamma
$
for all  $t\in [0,t_1)$.
By continuity and the assumption
$t_1<t_\ast$ this inequality also holds
for $t=t_1$. Hence
\begin{equation*}
|x(t_1)|+|y(t_1)|_\gamma\le |x_0|+(M_2+M_3)|y_0|_\gamma
\le (1+M_2+M_3)N\delta<\rho/2,
\end{equation*}
provided
$\delta\le {\rho}/[2N(1+M_2+M_3)]$.
This contradicts the definition of $t_1$ and
we conclude that $t_1=t_\ast$.

In the following, we assume that $\delta\le {\rho}/[2N(1+M_2+M_3)]$.
Then the estimates derived above and \eqref{estimate-phi} yield the
uniform bounds
\begin{equation}
\no v\no_{\EE_1(a)}+\sup_{t\in [0,a)}|v(t)|_\gamma
\le M,
\end{equation}
for every initial value $v_0\in B_{X_\gamma}(0,\delta)$ and every
$a<t_\ast$.
It follows from Corollary~3.2
in \cite{Pru03}
that the solution $v(t)$ of \eqref{v-equation} exists on $\R_+$.
\medskip\\
(f)
By repeating the above estimates on the interval $(0,\infty)$ we obtain
the estimates
\begin{equation}
\label{y-gamma-infty} |x(t)|\le |x_0|+M_3|y_0|_\gamma,\quad
|y(t)|_\gamma\le M_2 e^{-\omega t} |y_0|_\gamma, \quad t\in
[0,\infty),
\end{equation}
for all $x_0\in B_{X^c}(0,N\delta)$ and
$y_0\in B_{X^s_\gamma}(0,N\delta)$.
Moreover,
\begin{equation*}
\lim_{t\ra\infty} x(t)= x_0 +\int_0^\infty T(x(s),y(s))ds=:x_\infty
\end{equation*}
exists since the integral is absolutely convergent.
Next observe that we in fact obtain exponential convergence
of $x(t)$ towards $x_\infty$, as
\begin{equation*}
\begin{aligned}
|x(t)-x_\infty|=&\left|\int_t^\infty T(x(s),y(s))\,ds\right| \\
&\le \beta \int_t^\infty |y(s)|_1\,ds  \\
&\le \beta \left (\int_t^\infty e^{-\omega s p^\prime
}\,ds\right)^{1/p^\prime}
\no e^{\omega s}y\no_{\EE_1(\infty)} \\
&\le M_4  e^{-\omega t} |y_0|_\gamma, \quad t\ge 0.
\end{aligned}
\end{equation*}
This yields existence of
\begin{equation*}
v_\infty:=\lim_{t\ra\infty} v(t)=\lim_{t\ra\infty} x(t)+\phi(x(t))+y(t)=x_\infty+\phi(x_\infty).
\end{equation*}
Clearly, $v_\infty$ is an equilibrium
for equation \eqref{v-equation}, and $v_\infty+u_\ast\in\cE$ is an equilibrium for \eqref{AF}.
Due to \eqref {estimate-phi}, \eqref{y-gamma-infty}
and the exponential estimate for $|x(t)-x_\infty|$
we get
\begin{equation}
\label{exponential-v}
\begin{aligned}
|v(t)-v_\infty|_\gamma
&=|x(t)+\phi(x(t))+y(t)-v_\infty|_\gamma \\
&\le |x(t)-x_\infty|_\gamma +|\phi(x(t))-\phi(x_\infty)|_\gamma +|y(t)|_\gamma \\
&\le ({C}M_4+M_2)e^{-\omega t}|y_0|_\gamma\\
&\le Me^{-\omega t}|P^sv_0-\phi(P^cv_0)|_\gamma\,,
\end{aligned}
\end{equation}
thereby completing the proof of the second part of
Theorem~\ref{th:1}. Concerning stability,
note that given $r>0$ small enough we may choose
$0<\delta\le r$ such that the solution starting in
$B_{X_\gamma}(u_*,\delta)$ exists on $\R_+$ and stays within $B_{X_\gamma}(u_*,r)$.
\end{proof}
\noindent
\begin{remarks}
(a)
Theorem~\ref{th:1} shows, given that situation,
that near $u_*$ the set of equilibria constitutes the (unique) center manifold for \eqref{u-equation}.

\medskip
(b)
It is worthwhile to point
out a slightly different way to obtain the
function $\phi$ used in the proof of
Theorem~\ref{th:1}.
Applying the projections $P^s$ and $P^c$
to the equilibrium equation \eqref{equilibrium-psi}
yields the following equivalent system of equations near $v=0$
\begin{equation}
\label{impli}
A_s z=P^sG(x+z),\quad A_c x =P^c G(x+z),
\end{equation}
with $z=P^s\psi(\zeta) $ and
$x=P^c\psi(\zeta)$.
Since $G(0)=G^\prime(0)=0$ and $A_s$ is invertible, by the implicit function theorem we may solve the first equation
for $z$ in terms of $x$,
i.e. there is a $C^1$-function
$\phi:B_{X^c}(0,\rho_0) \rightarrow X^s_1$
such that
\begin{equation*}
\phi(0)=0\quad\text{and}\quad
A_s\phi(x)=P^sG(x+\phi(x)), \quad x\in B_{X^c}(0,\rho_0).
\end{equation*}
As $x+\phi(x)$ is the unique solution of the first
equation in \eqref{impli} we additionally have
 $A_c x=P^cG(x+\phi(x))$,
as well as $P^s\psi(\zeta)=\phi(P^c \psi(\zeta))$ for all $\zeta\in U$.
Since $G^\prime(0)=0$ we obtain
$A_s\phi^\prime(0)=P^sG^\prime(0)=0$ and this implies
$\phi^\prime (0)=0$.
This shows that $\cE\subset \cM$
with $\M=\{x+\phi(x)+u_\ast\st x\in B_{X^c}(0,\rho_0)\}$
in a neighborhood of $u_\ast$ in $X_1$.

$\cM$ is a $C^{1}$-manifold
of dimension $\ell:=\text{dim}\,(X^c)$
with tangent space $T_{u_\ast}(\cM)=X^c$
and $\cE$ is a submanifold in $\cM$.
In general, $\cE$ has lower dimension than
$\cM$. Our assumptions in Theorem~\ref{th:1}
do in fact exactly amount to asserting
that $\cE$ and $\cM$ are of equal dimension.
Since $\cE\subset\cM$ we can then conclude that
they coincide in a neighborhood of $u_\ast$.

\medskip
(c)
An inspection of the argument given above
shows that in fact all equilibria
of equation~\eqref{u-equation}
that are close to the equilibrium $u_\ast$
are contained
in a manifold
$\cM=\{x+\phi(x)+u_\ast:x\in B_{X^c}(0,\rho_0)\}$
such that $\phi(0)=\phi^\prime(0)=0$,
with no additional assumptions on the structure
of the equilibria.
To see this, let us once more consider the equation
\begin{equation}
\label{D}
A_s z=P^sG(x+z),\quad x\in X^c,\quad z\in X^s_1\,.
\end{equation}
Clearly, $x=z=0$ is a solution. Exactly as in the remark above, we
can solve \eqref{D} by the implicit function theorem for $z$ in
terms of $x$, obtaining a $C^1$-function $\phi:B_{X^c}(0,\rho_0)
\rightarrow X^s_1$ with $\phi(0)=\phi^\prime(0)=0$. If $v\in X_1$ is
an equilibrium for the evolution equation \eqref{v-equation} close to
$0$, then the pair $x=P^cv$, $z=P^sv$ necessarily satisfies equation
\eqref{D}, and therefore lies on the graph of $\phi$.

\medskip
(d)
We  illustrate by means of examples
that convergence to equilibria fails if one of the
conditions (i)-(iii) in Theorem~\ref{th:1}
does not hold.

\bigskip

\noindent {\bf Example 1.} Consider in $G:=\R^2\setminus\{0\}$ the
ODE system
\begin{equation}
\label{Ex1}
\begin{array}{r@{\,=\,}l}
\dot{x} & (x+y)(1-\sqrt{x^2+y^2}),\\
\dot{y} & (y-x)(1-\sqrt{x^2+y^2}).
\end{array}
\end{equation}
In polar coordinates (\ref{Ex1}) reads as
\[
\begin{array}{r@{\,=\,}l}
\dot{r} & -r(r-1), \\
\dot{\theta} & r-1,
\end{array}
\]
thus the set of equilibria $\cE$ of (\ref{Ex1}) in G is the unit circle,
and for any initial value $(x_0,y_0)\in G$ we have $r(t)\to 1$ as
$t\to \infty$. Since the phase portrait is rotationally invariant we
may restrict the stability analysis for $\cE$ to one equilibrium,
say $u_*=(0,1)$. Denoting the right-hand side of (\ref{Ex1}) by
$F(x,y)$, we have $F\in C^1(G)$, and
\[
A_0=-F'(u_*)=\left[
\begin{array}{c@{\;\;}c}
  0 & 1 \\
  0 & 1 \\
\end{array}
\right].
\]
The eigenvalues of $A_0$ are $0$ and $1$ with eigenvectors $(1,0)$
and $(1,1)$, respectively. Thus $0$ is semi-simple, and $N(A_0)$
coincides with the tangent space $T_{u_*}(\cE)$. Consequently, $u_*$ is
normally stable, and hence we can apply Theorem~\ref{th:1}
to conclude that each trajectory converges to
some point on the unit circle as $t\to\infty$.
It is readily seen that the trajectories satisfy
the relation $\theta(r)=c_0-\ln r$ for some appropriate constant
$c_0$, and this confirms that $\theta(r)$ converges as $r\to 1$.

\bigskip

\noindent {\bf Example 2.} In this example, we consider in
$G:=\R^2\setminus\{0\}$ the ODE system
\begin{equation}
\label{Ex2}
\begin{array}{r@{\,=\,}l}
\dot{x} & -x(\sqrt{x^2+y^2}-1)^3-y(\sqrt{x^2+y^2}-1)^m,\\
\dot{y} & -y(\sqrt{x^2+y^2}-1)^3+x(\sqrt{x^2+y^2}-1)^m,
\end{array}
\end{equation}
with $m=1$. In polar coordinates (\ref{Ex2}) reads as
\[
\begin{array}{r@{\,=\,}l}
\dot{r} & -r(r-1)^3, \\
\dot{\theta} & (r-1)^m.
\end{array}
\]
Again, the set of equilibria $\cE$ of (\ref{Ex2}) in $G$ is the unit
circle. As above, we may restrict the stability analysis for $\cE$
to one equilibrium, say $u_*=(0,1)$. Denoting the right side of
\eqref{Ex2} by $F(x,y)$ we obtain (in case $m=1$)
\[
A_0=-F'(u_*)=\left[
\begin{array}{c@{\;\;}c}
  0 & 1 \\
  0 & 0 \\
\end{array}
\right].
\]
Clearly, $\{0\}$ is an eigenvalue of $A_0$
with algebraic multiplicity $2$,
and $N(A_0)=\text{span}\{(1,0)\}$.
Therefore, the eigenvalue $\{0\}$
has geometric multiplicity $1$ and algebraic
multiplicity $2$.
So we have the following situation:
\begin{itemize}
\item[(i)]
$\cE=\{(x,y)\st x^2+y^2=1\}$ is a smooth manifold
of dimension $1$ in $\R^2$,
\item[(ii)]
the tangent space of $\cE$
at $u_\ast$ is given by $N(A_0)$,
\item[(iii)]
$\{0\}$ is {\sl not} semi-simple,
\end{itemize}
and hence condition (iii) of Theorem~\ref{th:1} is not satisfied. We
will show that the trajectories of system \eqref{Ex2} still converge
towards the unit circle, but will spiral around the circle at
increasing speed as $r\to 1$. This can be seen as follows. First we
observe that $V(x,y):=(r-1)^2$ with $r=\sqrt{x^2+y^2}$ is a Lyapunov
function for system \eqref{Ex2}, since for every solution $(x,y)$ of
\eqref{Ex2} we have
$$
\frac{d}{dt}V(x,y)=2\dot r(r-1)=-2r(r-1)^4\le 0.
$$
So $r(t)\to 1$ as $t\to\infty$ for every solution.
On the other hand one verifies that
the trajectories satisfy the relation
$\theta(r)=c_0+\ln({|r-1|}/{r})+{1}/(r-1).$
This shows that all trajectories spiral
around $\cE$ with increasing speed, in clockwise direction
as $r\nearrow 1$, and in counter-clockwise direction as
$r\searrow 1$.
\medskip\\
\noindent {\bf Example 3:}
Here we consider system \eqref{Ex2} with $m=2$.
This example is similar to the one in
\cite[p. 4]{Aul84}.
In this case we have
\[
A_0=-F'(u_*)=\left[
\begin{array}{c@{\;\;}c}
  0 & 0 \\
  0 & 0 \\
\end{array}
\right].
\]
Clearly, $\{0\}$ is now an eigenvalue with geometric multiplicity
$2$, and $N(A_0)=\R^2$. So condition (ii) of Theorem~\ref{th:1} is
not satisfied. The function $V$ from the previous example is again a
Lyapunov function, and this yields $r(t)\to 1$ as $t\to\infty$. The
trajectories satisfy $\theta(r)=c_0-\ln(|r-1|/r),$ showing that they
spiral counter-clockwise with increasing speed around the unit
circle as $r\to 1$.

\bigskip
\end{remarks}

\section{Quasilinear parabolic problems with nonlinear boundary conditions}
The analysis in the previous section applies in particular to quasilinear parabolic systems of partial differential equations with
{\em linear autonomous} boundary conditions. In this section we show how this can be extended to the case where also the boundary conditions
are nonlinear. For this purpose, let $\Omega\subset\R^n$ be an open bounded domain with boundary $\partial \Omega\in C^{2m}$. The outer normal at a
point $x\in\partial \Omega$ will be denoted by $\nu(x)$. Consider the problem
\begin{equation}
\label{eq:1}
\left\{
\begin{aligned}
\partial_t u(t)+A(u(t))u(t)&=F(u(t)) &\text{in }& \Omega, \\
  B_j(u(t))&=0  &\text{on }& \partial \Omega,\quad j\in\{1,\cdots,m\},\\
u(0)&=u_{0} &\text{in }& \Omega.
\end{aligned}
\right.
\end{equation}
Here we employ the maps
\begin{align}
[A(u)v](x)= & \sum_{|\alpha|= 2m} a_\alpha(x,u(x), \nabla u(x),\cdots,
     \nabla^{2m-1}u(x)) \,D^\alpha v(x), \quad x\in\Omega,\nonumber\\
[F(u)](x)= & \,f(x,u(x),\nabla u(x), \cdots,
     \nabla^{2m-1}u(x)),\quad x\in \Omega,\label{eq:1.1}\\
[B_j(u)](x)= &  \,b_{j}(x,u(x),\nabla u(x), \cdots,
     \nabla^{m_j}u(x)),\quad x\in\partial \Omega,\nonumber
\end{align}
for functions   $u\in BC^{2m-1}(\overline{\Omega};\C^N)$,
 and $v\in W^{2m}_p(\Omega;\C^N)$. The numbers $m_j$ are integers
strictly smaller than $2m$, and with $E=\C^N$, the coefficients are subject to the following regularity assumptions
\begin{itemize}
\item[{\bf (R)}] $a_\alpha \in C^1(E\times E^N\times  \cdots \times
E^{N^{2m-1}};BC(\overline{\Omega};\cB(E)))$
 for  $|\alpha|= 2m,$\\
$f\in C^1(E\times  E^N\times\cdots \times E^{N^{2m-1}};BC(\overline{\Omega};E)),$\\
$b_{j}\in C^{2m+1-m_j} (\partial\Omega  \times E \times E^N\times
\cdots \times
 E^{N^{m_j}}; E) \;$ for   $j\in\{1,\cdots,m\}$.
\end{itemize}
We set $B=(B_1,\cdots,B_m)$. We point out that, for a fixed
$u_0\in BC^{2m-1}(\overline{\Omega};\C^N)$, $A(u_0)$ is a linear
differential operator of order $2m$ with bounded coefficients;
whereas $F$ contains all terms involving derivatives of order $|\alpha|<2m$.

We will employ the $L_p$-setting for this problem as in \cite{LPS06}, hence we fix $p>n+2m$ and the basic spaces
\begin{equation*}
X_0=L_p(\Omega;E),\quad X_1 = W^{2m}_p(\Omega;E),
\quad X_\gamma=(X_0,X_1)_{1-1/p}=W^{2m(1-1/p)}_p(\Omega;E).
\end{equation*}
As in Section 2 we denote the norm in $X_j$ by $|\cdot|_j$ and open
balls in $X_j$ by $B_{X_j}(u,r)$, $j=0,1,\gamma$. Note that by the
Sobolev embedding theorem we have $X_\gamma\hookrightarrow
BC^{2m-1}(\overline{\Omega};E)$, which allows us to plug in
functions $u\in X_\gamma$ into the coefficients of $A$, into $f$ and
into the functions $b_j$ pointwise, without any growth restrictions
on these nonlinearities.

Assume we have a $C^1$-manifold of equilibria
$\Psi:U\ra X_1$ where $U\subset\R^k$ is an open neighborhood of $0$,
\begin{equation}
\label{Psi}
\begin{aligned}
A(\Psi(\zeta))\Psi(\zeta)&=F(\Psi(\zeta)) &\text{in }& \Omega, \; \zeta\in U,\\
B(\Psi(\zeta))&=0,&\text{on }&\partial \Omega, \; \zeta\in U,
\end{aligned}
\end{equation}
and set $u_*=\Psi(0)$. Assume that the rank of $\Psi^\prime(0)$ is
$k$ and that there are no other equilibria near $u_*$ in $X_1$,
i.e.\ $\cE\cap B_{X_1}(u_*,r_1)=\Psi(U)$, for some $r_1>0$, where as
in Section~2, $\cE$ denotes the set of equilibria of (\ref{eq:1}).

The linearization of (\ref{eq:1}) at $u_*$ is given by the operator $A_0$ defined as follows
\begin{equation}
\label{A*}
\begin{aligned}
&A_*v = A(u_*)v + (A^\prime(u_*)v)u_* - F^\prime(u_*)v&\quad\text{in } &\Omega, \\
&B_*v=B^\prime(u_*)v &\quad\text{on } &\partial\Omega, \\
&\mbox{with } v\in D(A_*)=D(B_*)=W^{2m}_p(\Omega;E),\\
&A_0=A_*|_{N(B_*)}.
\end{aligned}
\end{equation}
Next we consider the property of maximal $L_p$-regularity for the pair $(A_*,B_*)$,
and in particular for the operator $A_0$.
For this we only need to consider the
principal parts of the corresponding differential operator and of the boundary operators, i.e.
\begin{equation*}
\begin{aligned}
&A_\#(x,D)= \sum_{|\alpha|=2m} a_\alpha(x,u_*(x),\ldots,\nabla^{2m-1}u_*(x))D^\alpha, \\
&B_{j\#}(x,D)=\sum_{|\beta|=m_j} i^{m_j}[\partial
b_j/\partial(\partial_x^\beta
u)](x,\ldots,\nabla^{2m-1}u_*(x))D^\beta,
\end{aligned}
\end{equation*}
for $j=1,\ldots,m.$ Note that we use the notation $D=-i\nabla$, hence $\nabla^\beta=i^{|\beta|} D^\beta$.
It is shown in \cite{DHP06} that normal ellipticity of $A_\#$ and the Lopatinskii-Shapiro condition for
 $(A_\#,B_\#)$ are necessary, and in \cite{DHP03} that they are
also sufficient for $L_p$-maximal regularity of $A_0$. These conditions read as follows.
\vspace{1mm}
\begin{itemize}
\item[{\bf (E)}]
\, {\em For all $x\in\bar{\Omega}$, $\xi\in \R^n$, $|\xi|=1$,
$\sigma(\cA_{\#}(x,\xi))\subset \C_+,$\\
i.e.\ $\cA(x,D)$ is {\bf normally elliptic}.}
\vspace{2mm}
\item[{\bf (LS)}]
\, {\em For all $x\in \partial \Omega$,
$\xi\in\R^n$, with $\xi\cdot \nu(x)=0$, $\lambda\in \overline{\C_+}$,
$\lambda\neq 0$, and $h\in E^m$, the system of ordinary
differential equations on the half-line
\begin{equation*}
\begin{aligned}
\lambda v(y)+\cA_{\#}(x,\xi+i\nu(x)\partial_y)v(y)&=0, && y>0,\\
\cB_{j{\#}}(x,\xi+i\nu(x)\partial_y)v(0)&=h_j,&& j=1,\ldots,m,
\end{aligned}
\end{equation*}
admits a unique solution $v\in C_0(\R_+;E)$.}\\
({\bf Lopatinskii-Shapiro condition.})
\end{itemize}
\medskip
\noindent Now assume that $u_*\in X_1$ is an equilibrium of
(\ref{eq:1}), and let conditions (R), (E), and (LS) be satisfied.
It was shown in \cite{LPS06} that (\ref{eq:1}) then  admits a local
strong solution in the $L_p$-sense for each initial value $u_0\in
X_\gamma$, provided the compatibility condition $B(u_0)=0$ holds and
$|u_0-u_*|_\gamma$ is sufficiently small.
The solution map $[u_0\mapsto u(t,u_0)]$ defines a
local semi-flow in $X_\gamma$ near $u_*$ on the
nonlinear phase-manifold
$$\cM=\{u\in X_\gamma:\, B(u)=0 \mbox{ on }
\partial \Omega\}.$$
In case the equilibrium $u_*$ is {\em hyperbolic}, i.e.\
$\sigma(A_0)\cap i\R=\emptyset$, it was moreover shown in
\cite{LPS06} that it is isolated and that it has the so-called {\em
saddle point property}, which means that the local semi-flow in
$\cM$ admits a unique stable and unstable manifold near $u_*$. We
refer to \cite{LPS06} for details as well as to \cite{Pru03} in the
case of linear boundary conditions.

Returning to our situation,
differentiating (\ref{Psi}) w.r.t.\ $\zeta$ we obtain for $\zeta=0$
\begin{align*}
A(u_*)\Psi^\prime(0) + [A^\prime(u_*)\Psi^\prime(0)]u_* -F^\prime(u_*)\Psi^\prime(0)&=0
& \text{in } &\Omega,\\
B_j^\prime(u_*)\Psi^\prime(0)&=0
&\text{on }&\partial \Omega,\; j=1,\ldots,m.
\end{align*}
This shows that the image of $\Psi^\prime(0)$ is contained in the
kernel $N(A_0)$ of $A_0$, and also that $T_{u_\ast}(\cE)$, the
tangential space of $\cE$ at $u_\ast$, is contained in $N(A_0)$. As
in Section~2 we assume now that $R(\Psi^\prime(0))=N(A_0)$, that the
eigenvalue $0$ of $A_0$ is semi-simple, and that the remaining
spectrum of $A_0$ is contained in the open right half-plane $\C_+$.
Note that by boundedness of $\Omega$ and compact embedding, the
spectrum of $A_0$ consists only of isolated eigenvalues of finite
algebraic multiplicity, anyway.  We can now state the main result of
this section.
\goodbreak
\begin{theorem}
\label{th:2} Let $2m+n<p<\infty$, let $\Omega\subset\R^n$ be an open
bounded domain with boundary of class $C^{2m}$, and let the spaces
$X_j$, $j=0,1,\gamma$, be defined as above. Suppose $u_*\in X_1$ is
an equilibrium of (\ref{eq:1}), and assume that  conditions (R),
(E), and (LS) are satisfied. Let $A_0$ defined in \eqref{A*} denote
the linearization of (\ref{eq:1}) at $u_*$, and suppose that $u_*$
is normally stable, i.e.\ assume that
\begin{itemize}
\item[(i)]
\, near $u_*$ the set of equilibria $\cE$ is a $C^1$-manifold in $X_1$ of dimension $k\in\N$,
\item[(ii)]
\, the tangent space for $\cE$ at $u_*$ is given by $N(A_0)$,
\item[(iii)]
\, $0$ is a semi-simple eigenvalue of $A_0$, i.e.\ $R(A_0)\oplus N(A_0)=X_0$,
\item[(iv)]
\, $\sigma(A_0)\setminus\{0\}\subset \C_+=\{z\in\C:\, {\rm Re}\, z>0\}$.
\end{itemize}
\noindent Then $u_*$ is stable in $X_\gamma$, and there exists
$\delta>0$ such that the unique solution $u(t)$ of (\ref{eq:1}) with
initial value $u_0\in X_\gamma$, satisfying
$|u_0-u_*|_{\gamma}<\delta$ and the compatibility condition
$B(u_0)=0$ on $\partial \Omega$,
 exists on $\R_+$  and converges exponentially fast in
 $X_\gamma$ to some $u_\infty\in\cE$ as $t\rightarrow\infty$.
\end{theorem}
\begin{proof}
\,(a)\, The proof is  similar to that of Theorem~\ref{th:1}.
It is based again on the reduction to normal form.
We use the notation introduced above and denote as in Section 2 by $P^s$
and $P^c$ the projections onto $X^s_0=R(A_0)$
resp.\ $X^c =N(A_0)$.
We first center (\ref{eq:1}) around $u_*$ by setting $\bar{u}=u-u_*$, and obtain the
following problem for $\bar{u}$.
\begin{equation}
\label{eq:2}
\left\{
\begin{aligned}
\partial_t \bar{u}+A_*\bar{u}&=G(\bar{u}) &\text{in } &\Omega, \\
B_*\bar{u}&=H(\bar{u})&\text{on } &\partial \Omega,\\
\bar{u}(0)=\bar{u}_0:&=u_0-u_* &\text{in } &\Omega.
\end{aligned}
\right.
\end{equation}
Here $G$ is defined as in Section 2, and
\begin{equation*}
H(\bar{u})=B_*\bar{u}-B(u_*+\bar{u})= -[B(u_*+\bar{u})-B(u_*)-B^\prime(u_*)\bar{u}].
\end{equation*}
Exactly as in the proof of Theorem~\ref{th:1} we obtain
a function
$\phi\in C^{1}(B_{X^c}(0,\rho_0),X^s_1)$
with $\phi(0)=\phi^\prime(0)=0$
such that the equilibrium equation
\begin{equation*}
\begin{aligned}
A_*\psi(\zeta)&=G(\psi(\zeta))&\text{in } &\Omega,\; \zeta\in U, \\
B_*\psi(\zeta)&=H(\psi(\zeta))&\text{on } &\partial \Omega,\; \zeta\in U, \\
\end{aligned}
\end{equation*}
 for \eqref{eq:2}
can equivalently be expressed by
\begin{equation}
\label{equilibrium-phi-2}
\begin{aligned}
 P^c A_*\phi(v)&=P^c G(v+\phi(v)), \\
P^sA_*\phi(v) &= P^sG(v+\phi(v)), \quad B_*\phi(v)=H(v+\phi(v)),\\
\end{aligned}
\end{equation}
for every $v\in B_{X^c}(0,\rho_0)$.
%
%
We can now introduce the normal form of (\ref{eq:1})
for the variables
$$
v:=P^c (u-u_*)=P^c \bar{u},\quad
w:=P^s(u-u_*)-\phi(P^c (u-u_*))=P^s\bar{u}-\phi(P^c \bar{u}),
$$
which reads as
\begin{equation}
\label{normalform}
\left\{
\begin{aligned}
\partial_t v &= T(v,w) &\text{in } &\Omega,\\
\partial_t w + P^sA_\ast P^sw&=R(v,w)&\text{in }&\Omega,\\
B_*w&=S(v,w)&\text{on }&\partial \Omega, \\
v(0)&=v_0,\ w(0)=w_0 &\text{in } &\Omega.
\end{aligned}
\right.
\end{equation}
Using the equilibrium equations in \eqref{equilibrium-phi-2} we can
derive, similarly as in Section 2, the following expressions for
$T$, $R$ and $S$:
\begin{equation*}
\label{RTS}
\begin{aligned}
 &&T(v,w)&=P^c\big(G(v+\phi(v)+w)-G(v+\phi(v))\big)-P^cA_\ast w, \\
 &&R(v,w)&=P^s\big(G(v+\phi(v)+w)-G(v+\phi(v))\big)-\phi^\prime(v)T(v,w),\\
  &&S(v,w)&=H(v+\phi(v)+w)-B_*\phi(v)\\
 &&      &=H(v+\phi(v)+w)-H(v+\phi(v)).
\end{aligned}
\end{equation*}
Clearly,
\begin{equation*}
\label{equilibrium-nonlinear}
R(v,0)=T(v,0)=S(v,0)=0,\quad v\in B_{X^c}(0,\rho_0).
\end{equation*}
Therefore we are in the same situation as in Section 2, except
that here the infinite dimensional part, i.e.\ the equation for
$w$, has a nonlinear boundary condition in case $S(v,w)\not\equiv0$.
\medskip\\
(b)\, Let $0<a\le \infty$ and define the following function spaces
on $(0,a)\times\Omega$:
$$ \E_1(a)=H^1_p((0,a);X_0)\cap L_p((0,a);X_1),
\quad \E_0(a)=L_p((0,a);X_0).$$
We also need spaces for the boundary values. For this purpose, we set with $\kappa_j=1-m_j/2m-1/2mp$
$$Y_0=L_p(\partial\Omega;\C^N),\quad Y_j=W^{2m\kappa_j}_p(\partial\Omega;\C^N),$$
and
$$\F(a)=\prod_{j=1}^m \F_j(a),\quad \F_j(a)=W^{\kappa_j}_p((0,a);Y_0)\cap L_p((0,a);Y_j),\quad j=1,\ldots, m .$$
Note that by trace theory we have
\begin{equation*}
\begin{split}
&\F_j(a)\hookrightarrow BC([0,a]; W^{2m\kappa_j-2m/p}_p(\partial\Omega;\C^N), \\
&W^{2m\kappa_j-2m/p}_p(\partial\Omega;\C^N))\hookrightarrow BC(\partial\Omega;\C^N)
\end{split}
\end{equation*}
by the condition $p>2m+n$ and since $m_j<2m$.
The spaces $\F_j(a)$ are the trace spaces on the lateral boundary $(0,a)\times\partial\Omega$ of $(0,a)\times\Omega$ for the
derivatives $D^\beta u$ of order $|\beta|=m_j$ for $u\in \E_1(a)$.
\medskip\\
\noindent
The basic solvability theorem for the fully inhomogeneous linear problem
\begin{equation}
\label{linproblem}
\left\{
\begin{aligned}
\partial_t u+A_*u &=f(t) &\text{in } &\Omega,&t>0,\\
B_*u &=g(t)&\text{on } &\partial\Omega,&t>0,\\
u(0)&=u_0 &\text{in }&\Omega
\end{aligned}
\right.
\end{equation}
in the $L_p$-setting reads as follows, see \cite{DHP06}.
\begin{proposition}
\label{DHP} Let $a<\infty$. The linear problem (\ref{linproblem})
admits a unique solution $u\in\E_1(a)$ if and only if $f\in
\E_0(a)$, $g\in \F(a)$, $u_0\in X_\gamma$, and the compatibility
condition $B_*u_0= g(0)$ holds. There is a constant $C=C(a)>0$ such
that the estimate
$$\no u\no _{\EE_1(a)}
\le C\big(|u_0|_\gamma+ \no f\no_{\EE_0(a)}+\no g\no_{\F(a)}\big)$$
holds for the solution $u$ of (\ref{linproblem}).
\end{proposition}
\noindent
We shall also need a variant of Proposition \ref{DHP} for the problem
\begin{equation}
\label{linprobl}
\left\{
\begin{aligned}
\partial_t w+P^sA_\ast P^sw &=f(t) &\text{in } &\Omega,&t>0,\\
B_*w &=g(t)&\text{on } &\partial\Omega,&t>0,\\
w(0)&=w_0 &\text{in }&\Omega
\end{aligned}
\right.
\end{equation}
on the half-line, where we assume  $w_0\in X^s_\gamma$ and
$f\in L_p(\R_+;X^s_0)$.
For this purpose we proceed as follows.
Suppose first that $u$ solves (\ref{linproblem}) with $u_0=w_0$.
Since $A_\ast P^c u=B_\ast P^cu=0$
we then conclude that $w=P^su$ solves problem \eqref{linprobl}.
Let
$u_1$ denote the solution of (\ref{linproblem}) with $A_*$
replaced by $A_*+1$. The spectrum of $A_0+1$ is contained in
$\C_+$, hence we may apply Proposition 3.1 of \cite{LPS06} to
obtain a uniform estimate  for $u_1$ in $\E_1(\infty)$. Then
$u_2=u-u_1$ solves the problem
$$\partial_tu_2+A_*u_2=u_1,\quad B_*u_2=0,\quad u_2(0)=0.$$
As $\sigma(A_s)\subset\C_+$, $A_s$ has maximal $L_p$-regularity on
the half-line, hence we obtain also a uniform estimate for
$P^su_2$ in $\E_1(\infty)$. These arguments yield
the following result.
\goodbreak
\begin{proposition}
\label{LP} Let $a\le \infty$. The linear problem (\ref{linprobl}) admits a unique solution
$w\in\E_1(a)\cap L_p((0,a);X^s_0)$ if and only if
$f\in L_p((0,a);X^s_0  )$, $g\in \F(a)$,
$w_0\in X^s_\gamma $, and the compatibility condition $B_*w_0= g(0)$ holds.
There is a constant $C_0>0$, independent of $a$, such that the estimate
\begin{equation*}
\no w\no_{\EE_1(a)}
\le C_0\big(|w_0|_\gamma+ \no f\no_{\EE_0(a)}+\no g\no_{\F(a)}\big)
\end{equation*}
holds for the solution $w$ of (\ref{linprobl}),
for all functions $f\in L_p((0,a);X^s_0)$, $g\in\F(a)$
and all initial values $w_0\in X^s_\gamma$.
\end{proposition}
\noindent
Proposition \ref{LP} remains valid when we
replace $w(t)$ by $e^{\sigma t}w(t)$, $f(t)$  by $e^{\sigma t}f(t)$ and
$g(t)$ by $e^{\sigma t}g(t)$ where $0<\sigma\le \omega$,
$\omega<\inf\{{\rm Re}\,\lambda:\lambda\in\sigma(A_s)\}$.
\medskip\\
(c)\, Next we consider the nonlinearities $R$, $T$, and $S$.
Since by assumption the functions $a_\alpha$ and $f$ are
in $C^1$ and $p>n+2m$,
it follows easily
via the embedding $X_\gamma\hookrightarrow BC(\bar{\Omega})$ that $A$ and $F$
are as in Section 2. Hence we obtain as there the estimates
\begin{equation*}
|T(v,w)|_0
\le C_2(\eta+r)|w|_1 + C_3|w|_1,
\quad v\in B_{X^c}(0,\rho),\; w\in B_{X^s_\gamma}(0,\rho)\cap X_1,
\end{equation*}
where $r=3\rho$ and
$C_3:=\no P^c A_\ast P^c\no_{\cB(X_1,X^c)}$.
Since $\phi^\prime(0)=0$
we can assume that $\rho_0$ was chosen so small that
$|\phi^\prime(w )|_{\cB(X^c,X_1^s)}\le \eta$ for
all $w\in B_{X^c}(0,\rho_0)$. With  this we obtain
\begin{equation*}
|R(v,w)|_0\le  C_4(\eta+r)|w|_1,
\quad v\in \bar B_{X^c}(0,\rho),\; w\in \bar B_{X^s_\gamma}(0,\rho)\cap X_1.
\end{equation*}
Observe that in contrast to the previous section the constant $C_3$ is no longer small since $P^l$, $l\in\{c,s\}$, and $A_*$ do not commute.
However, this does not alter our conclusions.
It is more involved to derive the estimates on $S$ needed
for Proposition~\ref{LP}. Fortunately, we can refer to
\cite[Proposition 3.3]{LPS06}.
This result implies
$$\no e^{\omega t}(H(\bar{u}_1)-H(\bar{u}_2))\no_{\F(a)}
\le \eta \no e^{\omega t}(\bar{u}_1-\bar{u}_2)\no_{\EE_1(a)},$$
 for all
$e^{\omega t}\bar{u}_1,e^{\omega t}\bar{u}_2\in \E_1(a)$ such that
$|\bar{u}_1(t)|_\gamma, |\bar{u}_2(t)|_\gamma\le r, \; t\in [0,a].$
Therefore, by possibly decreasing  $r>0$,
with $\bar{u}_1=v+\phi(v)+w$ and $\bar{u}_2=v+\phi(v)$ this yields
$$
\no e^{\omega t} S(v,w)\no_{\F(a)}\le \eta \no e^{\omega t} w\no_{\EE_1(a)},
$$
for all $e^{\omega t} v,e^{\omega t} w\in \E_1(a)$ such that
$v([0,a])\subset B_{X^c}(0,\rho)$ and $w([0,a])\subset
B_{X^s_\gamma}(0,\rho)$. These are the estimates we need for
applying Proposition \ref{LP}.
\medskip\\
\noindent
(d)\, We may now follow part (c)-(f) of the proof of
Theorem~\ref{th:1} to complete the proof of Theorem~\ref{th:2},
the needed local well-posedness result
being Proposition 4.1 in \cite{LPS06}.
\end{proof}
\section{Convergence of solutions for the Mullins-Sekerka problem}
We consider the two-phase
quasi-stationary Stefan problem with surface tension, which has
also been termed Mullins-Sekerka model (or Hele-Shaw model with
surface tension) and is a model for phase transitions in
liquid-solid systems. Let $\Omega$ be a bounded domain in $\R^n$,
$n\ge 2$, with smooth boundary $\partial\Omega$. Let
$\Gamma_0\subset\Omega$ be a compact connected hypersurface
in $\Omega$
which is the boundary of an open set $\Omega_0\subset\Omega$,
and let
$\Gamma(t)$ be its position at time $t\ge 0$. Denote by $V(t,\cdot)$
and $\kappa(t,\cdot)$ the normal velocity and the mean curvature
of $\Gamma(t)$, and let $\Omega_1(t)$ (liquid phase) and
$\Omega_2(t)$ (solid phase) be the two regions in $\Omega$
separated by $\Gamma(t)$, with $\Omega_1(t)$ being the interior
region. Let further $\nu(t,\cdot)$ be the outer unit normal field
on $\Gamma(t)$ with respect to $\Omega_1(t)$. We shall use the
convention that $V$ is positive if $\Omega_1(t)$ is expanding, and
that the mean curvature is positive for uniformly convex
hypersurfaces. The {\em two-phase Mullins-Sekerka problem}
consists in finding a family $\Gamma(t)$, $t\ge 0$, of
hypersurfaces satisfying
\begin{equation} \label{MSP}
V =[\partial_\nu u_\kappa],\;\;t>0,\quad \Gamma(0)=\Gamma_0,
\end{equation}
where $u_\kappa=u_\kappa(t,\cdot)$ is, for each $t\ge 0$, the
solution of the elliptic boundary value problem
\begin{equation}
\label{Ex3}
\left\{
\begin{aligned}
\Delta u \,&= 0\quad \mbox{in}\;\Omega_1(t)\cup \Omega_2(t), \\
u \,&= \kappa \quad \mbox{on}\;\Gamma(t),\\
\partial_\nu u \,&= 0 \quad \mbox{on}\;\partial\Omega.
\end{aligned}
\right.
\end{equation}
Here $[\partial_\nu u_\kappa]:=\partial_\nu u_\kappa^2-\partial_\nu
u_\kappa^1$ stands for the jump of the normal derivative of
$u_\kappa$ across the interface $\Gamma(t)$, and $\partial_\nu u$
denotes the normal derivative of $u$ on $\partial \Omega$.

Assuming connected phases and that the interface does not touch
the fixed boundary $\partial\Omega$, the set of equilibrium states
of (\ref{MSP}), (\ref{Ex3}) consists precisely of all spheres
$S_R(x_0)\subset \Omega$, where $R$ denotes the radius and $x_0$
the center. Thus there is an $(n+1)$-parameter family of
equilibria, the parameters being the $n$ coordinates of the center
$x_0$ and the radius $R$.

Let now $\Sigma\subset\Omega$ be some fixed sphere without
boundary contact. We are interested in the asymptotic properties
of solutions of the Mullins-Sekerka problem that start in a
neighbourhood of $\Sigma$, that is $\Gamma_0$ is close to
$\Sigma$. Following \cite{EsSi98} we first use Hanzawa's method to
transform the original problem to a system of equations on a fixed
domain. Here the basic idea is to represent the moving interface
$\Gamma(t)$ as the graph of a function in normal direction of a
fixed reference surface, which will be $\Sigma$ in our case.
Denoting the parameterizing function by $\rho(t,\cdot)$ this leads
to a problem on $\Sigma$ of the form
\begin{equation} \label{TMS}
\dot{\rho}+B(\rho)S(\rho)=0,\;\;t>0,\quad \rho(0)=\rho_0,
\end{equation}
where $S(\rho)$ is the solution of the transformed elliptic boundary
value problem
\begin{equation}
\label{TMS2}
\left\{
\begin{aligned}
A(\rho) v \,&= 0 &\mbox{in}&\;\Omega_1\cup \Omega_2, \\
v \,&= K(\rho)  &\mbox{on}&\;\Sigma,\\
\partial_\nu v \,&= 0 &\mbox{on}&\;\partial\Omega.
\end{aligned}
\right.
\end{equation}
Here $\Omega_1$ and $\Omega_2$ are the two regions in $\Omega$
separated by $\Sigma$, with $\Omega_1$ being enclosed by $\Sigma$.
By construction, the solution $\Gamma(\cdot)\equiv \Sigma$ of the
original problem (\ref{MSP}) corresponds to the solution $\rho\equiv
0$ of (\ref{TMS}). The operator $F(\cdot):=B(\cdot)S(\cdot)$ in
(\ref{TMS}) is a nonlocal pseudo-differential operator of third
order and renders (\ref{TMS}) a quasilinear parabolic problem, see
\cite{EsSi98} for its precise definition and more details.

We want to study (\ref{TMS}) in an $L_p$ setting.
Let $p>n+2$
and define
\begin{equation*}
X_0=W^{1-1/p}_p(\Sigma), \quad X_1=W^{4-1/p}_p(\Sigma).
\end{equation*}
Given $J=(0,a)$, $a>0$, we view
(\ref{TMS}) as an evolution equation in the space $\EE_0(J)=L_p(J;X_0)$,
that is we are interested in solutions of (\ref{TMS}) in the class
$\EE_1(J)=H^1_p(J;X_0)\cap L_p(J;X_1)$.
For the corresponding trace space we have
$$
X_\gamma:=(X_0,X_1)_{1-1/p,p}=W^{4-4/p}_p(\Sigma).
$$
Note that, by Sobolev embedding, we
have $X_1\hookrightarrow X_\gamma\hookrightarrow C^2(\Sigma)$.

Let
\[
\mathfrak{A}:=\{\rho\in C^2(\Sigma): |\rho|_{C(\Sigma)}<\eta\}
\]
with $\eta>0$ sufficiently small denote the set of admissible
parameterizations. Setting $U:=X_1\cap \mathfrak{A}$, one has
$F\in C^\infty(U;X_0)$ and the linearization $L:=F'(0)$ is given
by
\[
L\rho=-[\partial_\nu TA_\Sigma \rho],
\]
where $Tg$ denotes the solution of the elliptic problem
\begin{equation}
\label{TMS3}
\left\{
\begin{aligned}
\Delta v \,&= 0\quad \mbox{in}\;\Omega_1\cup \Omega_2, \\
v \,&= g \quad \mbox{on}\;\Sigma,\\
\partial_\nu v \,&= 0 \quad \mbox{on}\;\partial\Omega,
\end{aligned}
\right.
\end{equation}
and
\[
A_\Sigma=-\frac{1}{n-1}\Big(\frac{n-1}{R^2}+\Delta_\Sigma\Big),
\]
with $R$ being the radius of $\Sigma$ and $\Delta_\Sigma$ the
Laplace-Beltrami operator on $\Sigma$. The operator $A_\Sigma$ is
the linearization $K'(0)$ of the transformed mean curvature operator
$K(\rho)$ at $\rho=0$. Concerning the linearization $L=F'(0)$ we
refer to \cite{EsSi98,PrSi06}.

One can show (cf. \cite[Theorem 2.1]{PrSi06}) that the spectrum of
$L$ consists of countably many real nonnegative eigenvalues of
finite algebraic multiplicity, and that $0$ is a semi-simple
eigenvalue of $L$ with multiplicity $n+1$, see also
\cite[Proposition 5.4 and Lemma 6.1]{EsSi98}. Moreover, the kernel
of $L$ is given by $N(L)=\mbox{span}\,\{Y_0,Y_1,\ldots,Y_n\}$, where
$Y_0\equiv 1$, and where $Y_j$, $1\le j\le n$, are the spherical
harmonics of degree 1. We may assume that $Y_j=R^{-1}p_j|_\Sigma$,
$1\le j\le n$, with $p_j$ being the harmonic polynomial of degree
$1$ given by $p_j(x)=x_j$ for $x\in \R^n$; by $p_j|_\Sigma$ we mean
the restriction of $p_j$ to $\Sigma$.

\medskip

Let us assume that $\Sigma$ is centered at the origin of $\R^n$.
Suppose ${\mathcal S}\subset \Omega$ is a sphere that is
sufficiently close to $\Sigma$. Denote by $(z_1,\ldots,z_n)$ the
coordinates of its center and let $z_0$ be such that $R+z_0$
corresponds to its radius. Then, by \cite[Section 6]{EsSi98}, the
sphere ${\mathcal S}$ can be parameterized over $\Sigma$ by the
distance function
\[
\rho(z)=\sum_{j=1}^n z_j Y_j-R+\sqrt{(\sum_{j=1}^n z_j
Y_j)^2+(R+z_0)^2-\sum_{j=1}^n z_j^2}.
\]
Denoting by $O$ a sufficiently small neighbourhood of $0$ in
$\R^{n+1}$, the mapping
$[z\mapsto \rho(z)]:O\to W^{4-1/p}_p(\Sigma)$
is smooth and the derivative at $0$ is given by
\[
\rho'(0)h=\sum_{j=0}^n h_j Y_j,\quad h\in \R^{n+1}.
\]
So we see that near $\Sigma$ the set $\cE$ of equilibria of
(\ref{TMS}) is a smooth manifold in $X_1$ of dimension $n+1$, and
that the tangent space $T_\Sigma(\cE)$ coincides with $N(L)$.

\medskip

In order to be able to apply Theorem~\ref{th:1} from Section 2 it remains
to verify that the operator $L$ has the property of maximal
$L_p$-regularity. This means we have to show that for any
$J=(0,a)$, $a>0$, and any $g\in \EE_0(J)$ the problem
\begin{equation} \label{maxregL}
\dot{\rho}+L\rho=g,\;\;t\in J,\quad \rho(0)=0,
\end{equation}
has a unique solution in the space $\EE_1(J)$. By means of the
standard localization method, perturbation arguments, and by
solving certain elliptic auxiliary problems, (\ref{maxregL}) can
be reduced to the following two-phase problem on $\R^n\times
\dot{\R}$ with $\dot{\R}=\R\setminus\{0\}$:
\begin{align}
-\Delta_x w-\partial_y^2 w & = 0, \quad t\in J,\,x\in \R^n,\,y\in
\dot{\R},\label{MSB1}\\
w|_{y=0}+\Delta_x \sigma & = 0,\quad t\in J,\,x\in \R^n,\label{MSB2}\\
\partial_t \sigma-[\partial_y w] & = h,\quad t\in J,\,x\in \R^n,\label{MSB3}\\
\sigma(0) & = 0,\quad x\in \R^n.\nonumber
\end{align}
Here $[\partial_y w]=\partial_y w|_{y=0^+}-\partial_y w|_{y=0^-}$,
and $h\in L_p(J;W^{1-1/p}_p(\R^n))$ is a given function. We take
the Fourier transform w.r.t. $x$ and denote the transformed
functions by $\tilde{w}$ and $\tilde{\sigma}$. Then (\ref{MSB1})
and (\ref{MSB2}) imply that
$\tilde{w}=e^{-|\xi|\,|y|}|\xi|^2\tilde{\sigma}$. Inserting this into
(\ref{MSB3}) leads to the subsequent problem for $\tilde{\sigma}$
on $\R^n$:
\[
\partial_t \tilde{\sigma}+2|\xi|^3
\tilde{\sigma}=\tilde{h},\;\;t\in J,\quad \tilde{\sigma}(0)=0.
\]
Set $Y_0=W^{1-1/p}_p(\R^n)$ and $Y_1=W^{4-1/p}_p(\R^n)$ and let
$G$ be defined by
$G=d/dt$ with domain $D(G)={}_0H^1_p(J;Y_0)$, here the zero means
vanishing trace at $t=0$. Then $G$ is sectorial, invertible and
admits an $\mathcal H^\infty$-calculus in $L_p(J;Y_0)$ of angle
$\pi/2$. Let further $D$ be the operator in $L_p(J;Y_0)$ with
symbol $2|\xi|^3$ and domain $D(D)=L_p(J;Y_1)$. Then $D$ is
sectorial and admits an $\mathcal H^\infty$-calculus in
$L_p(J;Y_0)$ of angle $0$. Thus by the Dore-Venni theorem, the
equation $G\sigma+D\sigma=h$ possesses a unique solution
$\sigma\in D(G)\cap D(D)$. Hence $L$ has the property of maximal
$L_p$-regularity.

So all assumptions of Theorem~\ref{th:1} are satisfied, hence we obtain the
following result, which is the main result in \cite{EsSi98} except
for the different functional analytic setting.
\begin{theorem}
\label{th:5}
Let $p>n+2$ and $\Omega\subset
\R^n$ be a bounded domain with boundary of class $C^2$. Suppose
$\Sigma$ is an arbitrary sphere in $\Omega$ of radius $R$ without
boundary contact. Then $\rho\equiv 0$ is a stable equilibrium of
(\ref{TMS}) in $X_\gamma=W^{4-4/p}_p(\Sigma)$, and there exists
$\delta>0$ such that if $|\rho_0|_{\gamma}<\delta$, then the
corresponding solution of (\ref{TMS}) exists globally and
converges at an exponential rate in $X_\gamma$ to
some equilibrium $\rho_\infty$ as $t\to\infty$.
In this sense, the sphere $\Sigma$ is a stable
equilibrium of the Mullins-Sekerka problem, and any solution
$\Gamma(\cdot)$ of (\ref{MSP}) that starts sufficiently close to
$\Sigma$ exists globally and converges to some sphere at
an exponential rate as
$t\to\infty$.
\end{theorem}
\noindent 
This approach can also be used to show the stability of
spheres for the two-phase quasi-stationary Stokes flow in a
bounded domain, see \cite{GP97,FR02} for alternate approaches
in the one-phase case.
Moreover, it can be applied to models in
tumor growth, see \cite{Cui07}
for a discussion of existing work.

\section{Stability of travelling wave
solutions to a quasilinear parabolic equation}
The situation of the generalized principle of linearized stability
may occur when studying the stability of travelling wave solutions
of parabolic equations, see e.g. \cite[Section 5.4]{Hen81} for the
semilinear case. In what follows we want to consider a quasilinear
variant of the Huxley equation:
\begin{equation} \label{TW1}
u_t-(\sigma(u_x))_x=f(u),\quad t>0,\,x\in \R.
\end{equation}
Here $f(r)=r(1-r)(r-a)$, $r\in \R$, where $a\in (0,1/2)$ is a
constant, and $\sigma$ is a $C^2$ smooth function on $\R$ satisfying
\begin{equation}
\label{TW2} 0<c_1\le \sigma'(r)\le c_2,\quad r\in \R.
\end{equation}
A travelling wave $u(t,x)=w(x+Vt)$ with speed $V$ satisfies
\begin{equation}
\label{TW3}
(\sigma(w^\prime(s)))^\prime-Vw^\prime(s)+f(w(s))=0,\quad s\in\R.
\end{equation}
Similarly to the special case $\sigma(r)=r$ (cf. \cite{Hen81}), one
can show, by means of a phase plane analysis, that for some $V>0$
(\ref{TW3}) admits a solution $w$ with $w(s)\to 0$ as $s\to -\infty$
and $w(s)\to 1$ as $s\to \infty$. For this purpose we introduce the
variable $z:=w^\prime$. Then (\ref{TW3}) is equivalent to the system
\begin{equation*}
\left\{
\begin{aligned}
\dot{w} & = z,\nonumber\\
\dot{z} & = \frac{1}{\sigma'(z)}\,(Vz-f(w))\nonumber.
\end{aligned}
\right.
\end{equation*}
Denoting its right-hand side by $H(w,z)$, we find that
\[
H'(i,0)=\left[
\begin{array}{c@{\;\;}c}
0 & 1 \\
-\frac{f'(i)}{\sigma'(0)} & \frac{V}{\sigma'(0)} \\
\end{array}
\right],\quad i\in \{0,1\}.
\]
We have $f'(0)=-a$ and $f'(1)=a-1$, thus the equilibria $(0,0)$ and
$(1,0)$ are both saddle points. The eigenvalues of $H'(0,0)$ are
given by
\[
\lambda_{1,2}=\frac{1}{2\sigma'(0)}\Big(V\pm
\sqrt{V^2+4a\sigma'(0)}\Big),
\]
$(1/\lambda_1,1)$ is an eigenvector to $\lambda_1>0$, thus the
unstable manifold points into the first and third quadrant, with
steeper slopes for higher values of $V\ge 0$.

Define the functions
\[
F(y)=\int_0^y f(r)dr,\quad G(y)=\int_0^y \sigma'(r)r\,dr,\quad y\in
\R.
\]
Then $(\ref{TW3})$ implies that
\begin{equation} \label{TW5}
\frac{d}{ds}\Big(G(z(s))+F(w(s))\Big)=Vz(s)^2,\quad s\in \R.
\end{equation}
In particular $G(z)+F(w)$ is a first integral if $V=0$. By
(\ref{TW2}) we further have $C_1 y^2\le 2G(y)\le C_2 y^2$, $y\in\R$.

We now consider the trajectory $\gamma$ that (near the origin) lies
on the unstable manifold to $(0,0)$ in the first quadrant. For
$V=0$, $\gamma$ cannot reach the line $w=1$, since $F(w)\le 0$ on
$\gamma$, and $F(1)=\frac{1}{6}(\frac{1}{2}-a)>0$. In case $V>0$,
(\ref{TW5}) shows that $\gamma$, as long as it remains in the first
quadrant, moves through increasing values $c$ of the level curves
$G(z)+F(w)=c$. For $V$ sufficiently large, $\gamma$ will reach the
level curve to $c=F(1)$ at some point with $z>0$. For continuity
reasons, there exists then $V>0$ for which $\gamma$ becomes a
heteroclinic orbit, connecting $(0,0)$ and $(1,0)$; observe that
$\gamma$ is the only such orbit. Hence there is a smooth solution
$w$ to (\ref{TW3}) satisfying $w'(s)>0$ for all $s\in \R$ and
\begin{equation}
\label{w-asymptotics}
\text{$(w(s),w'(s))\to (0,0)$ as $s\to -\infty,\ \ $
$(w(s),w'(s))\to (1,0)$ as $s\to \infty$,}
\end{equation}
both exponentially fast.
Clearly each translate $w(\cdot+\alpha)$ with $\alpha\in\R$ enjoys
the same properties.

In order to investigate the stability of the travelling wave we
change to moving coordinates with $y:=x+Vt$ and
$\tilde{u}(t,y):=u(t,x)$, thereby transforming (\ref{TW1}) into
\begin{equation} \label{TW6}
\tilde{u}_t-(\sigma(\tilde{u}_y))_y+V\tilde{u}_y=f(\tilde{u}),\quad
t>0,\; y\in \R.
\end{equation}
Evidently, $\tilde{u}=w$ and all translates of it are equilibria of
(\ref{TW6}). Letting $v:=\tilde{u}-w$ be the deviation from $w$, the
equation for $v$ reads as
\begin{equation}
\label{TW7}
{v}_t-(\sigma({v}_y+w_y))_y+V({v}_y+w_y)=f(v+w),\quad t>0,\; y\in \R.
\end{equation}
For this equation, the set
\begin{equation}
\label{E}
\cE=\{w(\cdot+\alpha)-w(\cdot):\alpha\in \R\}
\end{equation}
forms a one-dimensional smooth manifold of equilibria. Observe
that $\cE$ contains all equilibria $w(y)$ of (\ref{TW7})
satisfying $w(y)\to 0$ as $|y|\to \infty$. By definition, the
travelling wave under study is stable, if $v= 0$ is stable for
(\ref{TW7}). Of course, this has to be understood in the sense of
a suitable functional analytic setting.

Let us choose again the $L_p$ setting to study (\ref{TW7}). Let
$1<p<\infty$, $X_0=L_p(\R)$, and $X_1=H^2_p(\R)$. Define
$M:X_1\rightarrow X_0$ by
\[
M(v)=-(\sigma({v}_y+w_y))_y+V({v}_y+w_y)-f(v+w),\quad v\in X_1,
\]
This definition makes sense since by \eqref{TW2}--\eqref{TW3} and
\eqref{w-asymptotics} $w_y$, $w_{yy}$, and $f(w)$ belong to
$L_p(\R)$.
The fact that $f(v+w)$ belongs to the space $L_p(\R)$
can be justified by observing that
$$
f(v(y)+w(y))=f(w(y))+\int_0^1 f^\prime(w(y)+\tau v(y))\,d\tau\,
v(y),\quad y\in\R.
$$
Since $v$ and $w$ belong to $C_0(\R)$ one readily verifies
that $f^\prime (v+w)$ is continuous and bounded, and this yields
the statement for $f(v+w)$.
The linearization $A_0:=M'(0)$ of $M$
at $v=0$ is given by
\[
A_0v=-(\sigma'(w_y)v_y)_y +Vv_y-f'(w)v, \quad v\in
D(A_0):=X_1.
\]
By (\ref{TW2}), $A_0$ is a uniformly elliptic operator with
smooth coefficients whose leading coefficient satisfies
$-\sigma'(w'(y))\to -\sigma'(0)$ as $|y|\to \infty$. Thus, by
\cite[Theorem 5.7]{DHP03}, $A_0$ enjoys the property of maximal
$L_p$-regularity.

Next observe that $\cE\subset H^3_p(\R)\hookrightarrow X_1$ and
that the tangent space for the manifold $\cE$ at $v=0$ coincides
with $\mbox{span}\,\{w'\}$, the span of $w'\in X_1$.
On the other hand, by differentiating (\ref{TW3})
we see that $A_0w'=0$. So to show
normal stability of $v=0$, it remains to prove that $0$ is a
simple eigenvalue of $A_0$ and that the remainder of the spectrum of
$A_0$ lies in $\{z\in\C:\mbox{Re}\,z>0\}$.
We proceed similarly as in \cite[p. 131]{Hen81}, generalizing the
proof given there to the quasilinear case. Suppose $\lambda$ with
$\mbox{Re}\,\lambda\le 0$ is an eigenvalue of $A_0$ with
eigenfunction $v\in X_1$, that is
\[
v''+2b(y)v'+\frac{f'(w(y))}{\sigma'(w'(y))}\,v+\frac{\lambda
v}{\sigma'(w'(y))}=0,\quad y\in \R,
\]
where
\[
b(y)=\frac{\sigma''(w'(y))w''(y)-V}{2\sigma'(w'(y))},\quad y\in \R.
\]
By studying the characteristic equation for the limits $y\to \pm
\infty$, one sees that there exist constants $\delta,\,C>0$ such
that
\[ |v(y)|\le Ce^{(\delta+V/\sigma'(0))y},\;\; y\le 0,\quad |v(y)|\le
Ce^{-\delta y},\;\;y\ge 0.
\]
Letting $\varphi(y)=v(y)\exp(\int_0^y b(r)\,dr)$, we have at least
$\varphi(y)=O(e^{-V|y|/(2\sigma'(0))})$ for $|y|\to\infty$, and
\begin{equation} \label{TW8}
\varphi''+\Big(\frac{\lambda
}{\sigma'(w'(y))}+\frac{f'(w(y))}{\sigma'(w'(y))}-b'(y)-b(y)^2\Big)\varphi=0,
\quad y\in\R.
\end{equation}
The function $\psi(y):=w'(y)\exp(\int_0^y b(r)\,dr)$ is strictly
positive, and since $A_0w'=0$, it satisfies
\begin{equation} \label{TW9}
\psi''+\Big(\frac{f'(w(y))}{\sigma'(w'(y))}-b'(y)-b(y)^2\Big)\psi=0,
\quad y\in\R.
\end{equation}
Combining (\ref{TW8}), (\ref{TW9}) yields
\[
\varphi''+\frac{\lambda
}{\sigma'(w'(y))}\,\varphi-\frac{\psi''}{\psi}\,\varphi=0.
\]
Multiplying this equation by $\bar{\varphi}$, integrating over $\R$,
and integrating by parts gives
\begin{equation*}
 \label{TW10}
 \begin{aligned}
\lambda\int_\R \frac{|\varphi(y)|^2}{\sigma'(w'(y))}\,dy
&=\int_\R - \frac{\bar\varphi(y)}{\psi(y)}
\big(\varphi''(y)\psi(y)-\varphi(y)\psi''(y)\big)\,dy \\
&=\int_\R
\psi(y)^2 \left\{\big[\big({\text{Re}\,\varphi(y)}/{\psi(y)}\big)'\,\big]^2
+\big[\big({\text{Im}\,\varphi(y)}/{\psi(y)}\big)'\,\big]^2
\right\}\,dy,
\end{aligned}
\end{equation*}
hence $\lambda=0$. In view of (\ref{TW8}) we may then suppose that
$\varphi$ is real and nonnegative, and so the last formula above implies
that $\varphi/\psi=v/w'$ is constant on $\R$. Hence
$N(A_0)=\mbox{span}\,\{w'\}$.

We now show that the essential spectrum $\sigma_e(A_0)$ of $A_0$,
that is the set of all spectral points except isolated eigenvalues
of finite multiplicity, is contained in $[{\rm Re}\,z>0]$. This can
be seen as follows.

By the asymptotics of $w$, see \eqref{w-asymptotics}, we know that
$f^\prime (w(y))\to f^\prime(0)=-a$ as $y\to-\infty$, and that
$f^\prime (w(y))\to f^\prime(1)=-(1-a)$ as $y\to \infty$ at an
exponential rate. Due to $a\in (0,1/2)$, $a$ is the smaller of the
two numbers $\{a,(1-a)\}$. Fix $\varepsilon>0$ so that
$a-\varepsilon>0$. We can then find a number  $R>0$ so that
$-f^\prime(w(y))\ge a-\varepsilon$ whenever $|y|\ge R$. Let $c$ be a
bounded continuous function on $\R$ that agrees with $-f'(w)$ on
$[|y|\ge R]$ and satisfies $c(y)\ge a-\varepsilon$ for all $y\in\R$.
Let $B:X_1\to X_0$ be the operator defined by
$$
Bv:=-(\sigma'(w_y)v_y)_y +Vv_y +cv.$$ One readily shows that $B$ is
accretive on $L_p(\R)$ and also that $\sigma(B)$, the spectrum of
$B$, is contained in $[{\rm Re}\,z\ge (a-\varepsilon)]$. Next, note
that $A_0$ can be written as $A_0=B+S$, where $S$ is a perturbation
which is relatively compact with respect to $B$.
For this we note that $S$ can be written as
$Sv=-(f^\prime(w) +c)\chi v$, where $\chi$ is a smooth cut-off function
for the interval $[-R,R]$, with support contained, say, in
$\Omega=(-2R,2R)$.
Since $H^s_p(\Omega)$ is compactly embedded in $L_p(\Omega)$
for $s>0$ we conclude that
$S$ is a compact operator from $H^s_p(\R)\to L_p(\R)$.
A well-known perturbation result now shows that
 $\sigma_e(A_0)$ must also be contained in
$[{\rm Re}\, z\ge (a-\varepsilon)]$. Since this is true for every
$\varepsilon$, we have proved that $\sigma_e(A_0)\subset[{\rm
Re}\,z\ge a]$. $A_0$ might still have isolated eigenvalues outside
of this set. As $A_0$ generates an analytic semigroup, they must be
contained in an appropriate sector with opening angle
$\theta<\pi/2$. Since we have already shown that there are no
eigenvalues in $[{\rm Re}\,z< 0]$ we conclude that there is a number
$\alpha>0$ so that $\sigma(A_0)\setminus\{0\}\subset [{\rm Re}\,z\ge
\alpha]$. Finally, since the operator $A_1$ defined by
\[
A_1
v:=v_{yy}+\Big(\frac{f'(w)}{\sigma'(w')}-b'(y)-b(y)^2\Big)v,\quad
v\in D(A_1)=X_1,
\]
is self-adjoint in $L_2(\R)$, it follows in view of (\ref{TW8}) with
$\lambda=0$ that the eigenvalue $0$ of $A_0$ is semi-simple.

Summarizing we have shown that
\begin{itemize}
\item the set $\cE$ consists of all equilibria of (\ref{TW7}) in $X_1$ and forms a smooth $1$-dimen\-sio\-nal manifold,
\item $T_0(\cE)={\rm span}\,\{w^\prime\}=N(A_0)$,
\item the eigenvalue $0$ is semi-simple,
\item $\sigma(A_0)\setminus\{0\}\subset \C_+=\{z\in\C:\, {\rm Re}\, z>0\}$,
\end{itemize}
and we can now formulate our main result for this section.
\begin{theorem}
Let $p\in (1,\infty)$, $f(r)=r(1-r)(r-a)$, $r\in \R$, with $a\in
(0,1/2)$, and $\sigma\in C^2(\R)$ such that (\ref{TW2}) holds.

Then (\ref{TW1}) possesses a travelling wave solution
$u(t,x)=w(x+Vt)$ with speed $V>0$ and profile $w\in C^3(\R)$
satisfying $w'(r)>0$, $r\in \R$, and $(w(r),w'(r))\to (0,0)$ as
$r\to -\infty$, as well as $(w(r),w'(r))\to (1,0)$ as $r\to
\infty$.

This travelling wave is stable in the sense that $v\equiv 0$ is a
stable equilibrium of (\ref{TW7}) in $X_\gamma=W_p^{2-2/p}(\R)$.
Moreover there exists $\delta>0$ such that if
$|v_0|_\gamma<\delta$, then the solution $v$ of (\ref{TW7}) with
$v(0)=v_0$ exists globally and converges at an exponential
rate in $X_\gamma$ to some equilibrium $v_\infty$, i.e. to some
element of the set $\cE$ defined in (\ref{E}). In this sense, any
solution $u$ of (\ref{TW1}) that starts sufficiently close to $w$
exists globally and converges at an exponential rate as $t\to
\infty$ to some translate $w(x+Vt+\alpha)$, $\alpha\in \R$, of the
travelling wave solution.
\end{theorem}
\noindent This approach applies to many other travelling wave
solutions of quasilinear parabolic systems, as soon as the
condition of normal stability is satisfied.
\section{Convergence for abstract quasilinear problems II}
\noindent
We return to the setting of Section 2 for the case that 
$\sigma(A_0)$ also contains an unstable part, i.e. we now assume that
\begin{equation}
\label{spectrum}
\sigma(A_0)=\{0\}\cup\sigma_s\cup \sigma_u,\quad
\text{with } \sigma_s\subset \C_+,\; \sigma_u\subset \C_-,
\end{equation}
such that $\sigma_u\ne\emptyset$.
In this situation we can prove the following result.
\goodbreak
\begin{theorem}
\label{th:3}
Let $1<p<\infty$.
Suppose
$u^*\in V\cap X_1$ is an equilibrium of (\ref{u-equation}),
and suppose that the functions  $(A,F)$ satisfy \eqref{AF}.
Suppose 
further that $A(u^*)$ has the property of maximal
$L_p$-regularity. 
Let $A_0$ be the linearization of (\ref{u-equation}) at $u_*$.
Suppose that $u_*$ is normally hyperbolic, i.e.\ assume that
\begin{itemize}
\item[(i)] near $u_*$ the set of equilibria $\cE$ is a $C^1$-manifold in $X_1$ of dimension $m\in\N$,
\vspace{-4mm}
\item[(ii)] \, the tangent space for $\cE$ at $u_*$ is given by $N(A_0)$,
\item[(iii)] \, $0$ is a semi-simple eigenvalue of 
$A_0$, i.e.\ $ N(A_0)\oplus R(A_0)=X_0$,\item[(iv)]
 \, $\sigma(A_0)\cap i\R =\{0\}$, $\sigma_u:=\sigma(A_0)\cap \C_-\neq\emptyset$.
\end{itemize}
Then  $u_*$ is unstable in $X_\gamma$ and even in $X_0$.\\
 For each sufficiently small $\rho>0$ there exists $0<\delta\le \rho$ 
such that the unique solution $u(t)$ of (\ref{u-equation})
with initial value $u_0\in B_{X_\gamma}(u_*,\delta)$ either satisfies 
\begin{itemize}
\item[$\bullet$]
${\rm dist}_{X_\gamma}(u(t_0),\cE)>\rho$ for some finite time $t_0>0$, or 
\item[$\bullet $]
$u(t)$ exists on $\R_+$ and converges at an exponential rate
to some $u_\infty\in\cE$ in $X_\gamma$ as $t\rightarrow\infty$. 
\end{itemize}
\end{theorem}
\begin{proof}
The first assertion follows from \cite[Theorem 6.2]{Pru03}, so we need to prove 
the second claim.\\
(a)
Let $P^l$ denote the spectral projections corresponding to the
spectral sets $\sigma_l$, 
where $\sigma_c=\{0\}$ and $\sigma_s,\sigma_u$ are 
as in \eqref{spectrum}.
Let $X^l_j=P^l(X_j)$, $l\in\{c,s,u\}$,
where these spaces are
equipped with the norms of $X_j$ for 
 $j\in\{0,1,\gamma\}$. 
We may assume that $X_1$ is equipped with the
graph norm of $A_0$, i.e. $|v|_1:=|v|_0+|A_0v|_0$ for $v\in X_1$.
 Since the operator $-A_0$ generates an analytic $C_0$-semigroup on $X_0$, $\sigma_u$ is a compact spectral set for $A_0$. 
This implies that $P^u(X_0)\subset X_1$.
Consequently, $X^u_0$ and $X^u_1$ coincide as vector spaces.
In addition, since $A_u$, the part of $A_0$ in $X^u_0$, is invertible, 
we conclude that the spaces $X^u_j$ carry equivalent norms. 
We set $X^u:=X^u_0=X^u_1$ 
and equip $X^u$ with the norm of $X_0$, that is, $X^u=(X^u,|\cdot|_0)$.  
As in the proof of Theorem~\ref{th:1}
we obtain the decomposition
\begin{equation*}
X_1=X^c\oplus X_1^s\oplus X^u ,\quad X_0=X^c\oplus X_0^s\oplus X^u,
\end{equation*}
and this decomposition reduces $A_0$ into 
$A_0=A_c\oplus A_s\oplus A_u$,
where  $A_l$  is the part of $A_0$ in $X^l_0$
for $l\in\{c,s,u\}$. 
It follows that $\sigma(A_l)=\sigma_l$ for $l\in\{c,s,u\}$.
Moreover, due to assumption (iii), $A_c\equiv 0$. In the sequel, as a norm in 
$X_j$ we take 
\begin{equation}
\label{norm}
|v|_j=|P^cv|+|P^sv|_j+|P^uv|\quad\text{for}\quad j=0,\gamma,1.
\end{equation} 
We remind that the spaces $X^l_j$ have been given the norm
of $X^l_0$ for $l\in\{c,u\}$.

We also fix constants $\omega\in (0,\inf {\rm Re}\,\sigma(-A_u))$ and $M_5>0$ such that
$|e^{A_ut}|\le  M_5e^{-\omega t}$ for all $t>0$. 
Wlog\ we may take $\omega\le 1$.
\medskip\\
(b)
Let $\Phi$ be the mapping obtained in step (b) of the proof
of Theorem~\ref{th:1}, and set
$\phi_l(x):=P^l\Phi(x)$ for $l\in\{s,u\}$ and for $x\in B_{X^c}(0,\rho_0)$.
Then
\begin{equation}
\label{phi-s-u}
\phi_l\in C^1(B_{X^c}(0,\rho_0), X^l_1),
\quad \phi_l(0)=\phi^\prime_l(0)=0\quad\text{for } l\in\{s,u\}.
\end{equation}
These mappings parametrize the manifold $\cE$ of equilibria
near $u_\ast$ via
$$
[x\mapsto (x+\phi_s(x)+\phi_u(x)+u_\ast)],\quad x\in B_{X^c}(0,\rho_0).
$$
We may assume that
$\rho_0$ has been chosen small enough so that
\begin{equation}
\label{small}
\no \phi^\prime_l(x)\no_{\cB(X^c,X^1_l)}\le 1,
\quad x\in B_{X^c}(0,\rho_0),\quad l\in\{s,u\}.
\end{equation}
(c) The equilibrium equation \eqref{equilibrium-psi} now corresponds
to the system
\begin{equation}
\label{equilibrium-xyz}
\begin{aligned}
 P^cG(x+\phi_s(x)+\phi_u(x))&=0,\\
 P^l G(x+\phi_s(x)+\phi_u(x))&=A_l\phi_l(x),
 \quad x\in B_{X^c}(0,\rho_0), \quad l\in \{s,u\}.
\end{aligned}
\end{equation}
The canonical variables are
$$x=P^c v,\quad y= P^sv-\phi_s(x),\quad z=P^uv-\phi_u(x)$$
and the canonical form of the system  is given by
\begin{equation}
\label{hypcanform}
\left\{
\begin{aligned}
\dot{x}&=T(x,y,z),& x(0)=x_0,\\
\dot{y}+A_sy&=R_s(x,y,z), & y(0)=y_0, \\
\dot{z}+A_uz&=R_u(x,y,z),& z(0)=z_0.
\end{aligned}
\right.
\end{equation}
Here the functions $T$, $R_s$, and $R_u$ are given by
\begin{equation}
\label{RTS-s-l}
\begin{aligned}
T(x,y,z)=& \,P^c\big(G(x+y+z+\phi_s(x)+\phi_u(x))-G(x+\phi_s(x)+\phi_u(x))\big),\\
R_l(x,y,z)=& \,P^l\big(G(x+y+z +\phi_s(x)+\phi_u(x))-G(x+\phi_s(x)+\phi_u(x))\big) \\
& -\phi_l^\prime(x)T(x,y,z),\quad l\in\{s,u\}, \\
\end{aligned}
\end{equation}
where we have used the equilibrium equations
\eqref{equilibrium-xyz}. Clearly,
\begin{equation*}
\label{RTS=0}
R_l(x,0,0)=T(x,0,0)=0,\quad x\in B_{X^c}(0,\rho_0),\quad l\in\{s,u\},
\end{equation*}
showing that the equilibrium set $\cE$ of (\ref{u-equation}) near $u_*$
has been reduced to the set
$B_{X^c}(0,\rho_0)\times\{0\}\times\{0\}\subset X^c\times X^s\times X^u$.
\smallskip\\
There is a unique correspondence between the solutions of (\ref{u-equation})
close  to $u_*$ in $X_\gamma$ and those of (\ref{hypcanform}) close to $0$.
We again call (\ref{hypcanform})
the {\em normal form} of (\ref{u-equation}) near
its {\em normally hyperbolic} equilibrium $u_*$.
\medskip\\
(d)
The estimates for $R_l$ and $T$ are similar to those derived in Section 2,
and we have
\begin{equation}
\label{estimate-Rl-T}
 |T(x,y,z)|,\ |R_l(x,y,z)|_0 \le \beta(|y|_1+|z|), 
\end{equation}
for all 
	$x,z\in \bar{B}_X^{\tilde l}(0,\rho)$,
	 $\tilde l\in\{c,u\}$, and $y\in \bar B_{X^s_\gamma}(0,\rho)\cap X_1$,
where $\rho\le\rho_0$, $r=5\rho,$  and $\beta=C_2(\eta+r)$.
\medskip\\
(e) Let us assume for the moment that 
$\rho$ is chosen so that $4\rho\le \rho_0$.
Let $u(t)=u_*+\Phi(x(t))+y(t)+z(t)$ be a solution 
of \eqref{hypcanform}
on some maximal time interval $[0,t_*)$ which satisfies 
${\rm dist}_{X_\gamma}(u(t),\cE)\le \rho$. Set
\begin{equation*}
t_1:=t_1(x_0,y_0,z_0):=\sup\{t\in (0,t_\ast)\st 
|u(\tau)-u_*|_\gamma \le 3\rho,\ \tau\in [0,t]\}
\end{equation*}
and suppose that $t_1<t_\ast$. Assuming wlog that the embedding constant 
of $X_1\hookrightarrow X_\gamma$ is less or equal to one 
it follows from \eqref{norm}, \eqref{small} and the definition of $t_1$ that
\begin{equation}
\label{XYZ}
|x(t)|, |y(t)|_\gamma, |z(t)|\le 3\rho,\quad t\in [0,t_1],
\end{equation}
so that the estimate \eqref{estimate-Rl-T}
holds for $(x(t),y(t),z(t))$, $t\in [0,t_1]$.

Since $\cE$ is a finite-dimensional manifold,
for each $u\in B_{X_\gamma}(u_*,3\rho)$ 
there is $\bar{u}\in\cE$ such that 
${\rm dist}_{X_\gamma}(u,\cE)=|u-\bar{u}|_\gamma$, and by the triangle 
inequality  $\bar{u}\in B_{X_\gamma}(u_*,4\rho)$. Thus we may write 
$u=u_*+\Phi(x)+y+z$ and $\bar{u}=u_*+\Phi(\bar{x})$, and therefore
\begin{align*}
\rho\ge{\rm dist}_{X_\gamma}(u,\cE)&=|u-\bar{u}|_\gamma\\
&=
|x-\bar{x}|+|y+\phi_s(x)-\phi_s(\bar{x})|_\gamma+|z+\phi_u(x)-\phi_u(\bar{x})|\\
&\ge |x-\bar{x}|+|z|-|\phi_u(x)-\phi_u(\bar{x})|\ge |z|,
\end{align*}
since $x,\bar x\in B_{X^c}(0,\rho_0)$ and $\phi_s$ is non-expansive,
see \eqref{small}.
Therefore we obtain the improved estimate
$|z(t)|\le  \rho$ for all $t\in [0,t_1]$.

We begin the estimates with that for the unstable component $z(t)$. Integrating the 
equation for $z$ backwards yields
\begin{equation}
\label{z-backwards}
z(t)= e^{A_u (t_1-t)}z(t_1)-
\int_t^{t_1}e^{A_u(s-t)} R_u(x(s),y(s),z(s))\,ds.
\end{equation}
With \eqref{estimate-Rl-T} and $|z(t_1)|\le \rho$
we get
\begin{equation*}
|z(t)|\le  M_5e^{-\omega(t_1-t)}\rho + \beta M_5\int_t^{t_1} e^{-\omega(s-t)}
(|y(s)|_1 + |z(s)|)\,ds
\end{equation*}
for $t\in [0,t_1]$.
Gronwall's inequality yields 
\begin{equation*}
\label{esthypz}
|z(t)|\le  M_5 e^{-\omega_1(t_1-t)}\rho + \beta M_5\int_t^{t_1}e^{-\omega_1(s-t)}
|y(s)|_1\,ds
\end{equation*}
for $t\in [0,t_1],$
where $\omega_1=\omega-\beta M_5>0$ provided $\beta$, i.e.\ $\eta,r$ are small enough.
In particular, with $M_6=M_5/\omega_1$, this inequality implies
\begin{equation}
\label{z-Lq}
\no z\no _{L_q(J_1;X_0)}\le  M_6\rho 
+\beta M_6\no y\no _{L_q(J_1;X_1)},
\end{equation}
where we have set $J_1=(0,t_1)$; here $q\in[1,\infty]$ is arbitrary at the moment. 
A similar estimate holds for the time-derivative of $z$, namely
\begin{equation}
\label{dotz-Lq}
\no \dot{z}\no _{L_q(J_1;X_0)}
\le  (\no A_u\no +\beta)\no z\no _{L_q(J_1;X_0)}+
\beta\no y\no _{L_q(J_1;X_1)}.
\end{equation}
Note that 
\begin{equation}
\label{HypEst-z}
\begin{aligned}
|z(t+h)-z(t)|&\le  h^{1/p^\prime}\no \dot{z}\no _{L_p(J_1;X_0)},
\\
\int_0^{t_1-h}|z(t+h)-z(t)|\,dt &\le  h \no \dot{z}\no_{L_1(J_1;X_0)}.
\end{aligned}
\end{equation}

Next we consider the equation for $x$. We have
\begin{align*}
|x(t)| &\le |x_0| +\int_0^t|\dot{x}(s)|\,ds
= |x_0|+\int_0^t|T(x(s),y(s),z(s))|\,ds\\
 &\le |x_0| +\beta( \no y\no_{L_1(J_1; X_1) } +\no z\no_{L_1(J_1;X_0)}).
\end{align*}
Combining this estimate with that for $z$ we obtain
\begin{equation*}
\begin{aligned}
\label{esthypx}
\sup_{t\in J_1}|x(t)| &\le |x_0| +\no \dot{x}\no _{L_1(J_1;X_0)},
 \\\no \dot{x}\no _{L_q(J_1;X_0)} 
&\le \beta (M_6\rho +
(1+\beta M_6)\no y\no_{L_q(J_1; X_1)} ).
\end{aligned}
\end{equation*}
This estimate is best possible and shows that in order to control $|x(t)|$ 
we must be able to control $\no y\no_{L_1(J_1; X_1)}$. Note that \begin{equation}
\label{HypEst-x}
\begin{aligned}
|x(t+h)-x(t)|&\le  h^{1/p^\prime}\no\dot{x}\no_{L_p(J_1;X_0)}, 
\\
\int_0^{t_1-h}|x(t+h)-x(t)|\,dt
&\le  h \no\dot{x}\no_{L_1(J_1;X_0)}.
\end{aligned}
\end{equation}
Now we turn to the equation for $y$, the stable but infinite dimensional 
part of the problem.
As in the proof of Theorem~\ref{th:1}, part (e), we obtain
from \eqref{estimate-Rl-T}
\begin{equation*}
\label{y-estimate-stable}
\no  y\no_{\EE_1(t_1)}
\le M_1|y_0|_\gamma + \beta M_0 (\no y\no_{\EE_1(t_1)}+\no z\no_{\EE_0(t_1)}).
\end{equation*}
Employing \eqref{z-Lq} with $q=p$ we get
\begin{equation*}
\label{y-estimate-stable-III}
\no  y\no_{\EE_1(t_1)}
\le M_1|y_0|_\gamma + \beta M_0 M_6\rho  + 
\beta M_0(1+ \beta M_6)\no y\no_{\EE_1(t_1)}.
\end{equation*}
Assuming $\beta M_0 (1+ \beta M_6))<1/2$, this yields
\begin{equation}
\label{y-estimate-stable-IV}
\no  y\no_{\EE_1(t_1)}
\le 2M_1|y_0|_\gamma + 2\beta M_0 M_6\rho.
\end{equation}
Repeating the estimates leading up to
\eqref{y-estimate-gamma} with $\sigma=0$ we now get
\begin{equation}
|y(t)|_\gamma \le C_5 
(|y_0|_\gamma + \beta \rho) ,\quad t\in [0,t_1],
\end{equation}
where $C_5$ is a constant independent of $\rho$, $y_0$
and $t_1$. 
In particular, we see that
$|y(t)|_\gamma\le  \rho $ for all $t\in J_1$, provided $|y_0|_\gamma$ and $\beta$, i.e.\ 
$\eta$ and $r$  are sufficiently small.
\medskip\\
For later purposes we need an estimate for $|y(t+h)-y(t)|_\gamma$. We have
\begin{equation}
\label{HypEst-y}
\begin{aligned}
|y(t+h)-y(t)|_\gamma&\le C |y(t+h)-y(t)|_0^{1-\gamma}|y(t+h)-y(t)|_1^\gamma\\
&\le  C h^{(1-\gamma)/p^\prime}\no \dot{y}\no ^{1-\gamma}_{L_p(J_1;X_0)}
(|y(t+h)|_1^\gamma+y(t)|_1^\gamma)
\end{aligned}
\end{equation}
for all $t\in [0,t_1]$, $t+h\in [0,t_1]$ 
with $y(t+h),y(t)\in X_1$. We remind that $\gamma=1-1/p$.

Unfortunately, this is not enough to keep $|x(t)|$ small on $J_1$, for this we need to control $\no y\no_{L_1(J_1; X_1)}$, and we cannot
expect maximal regularity in $ L_1$.  

To handle $\no y\no_{L_1(J_1; X_1)}$, 
we are forced to use another type of maximal regularity, namely that for the vector-valued Besov
spaces $B^\alpha_{1\infty}(J_1;X)$, where $\alpha\in(0,1)$; 
cf.\ \cite[Theorem 7.5]{Pru93}. 
Before stating the result we remind that
\begin{equation*}
\begin{split}
&\no g\no_{B^\alpha_{1\infty}(J;X)}
:=\no g\no_{L_1(J;X)}+[g]_{J;\alpha,X}, \\
&[g]_{J;\alpha,X}:=\sup_{0<h\le \min(1,a)}
h^{-\alpha}\int_0^{a-h} |g(t+h)-g(t)|_X\,dt
\end{split}
\end{equation*}
defines a norm for $g\in B^\alpha_{1\infty}(J;X)$,
where $J=(0,a)$.
The maximal regularity result, 
which is valid for all exponentially stable analytic $C_0$-semigroups, 
reads as follows: 
there is a constant $M_{7}$ depending only on 
$A_s$ and on $\alpha\in(0,1)$ such that the solution $y$ of 
\begin{equation}
\label{SCP}\dot{y}+A_sy =f,\quad t\in J, \quad y(0)=y_0,
\end{equation}
satisfies the estimate
$$
\no y\no _{B^\alpha_{1\infty}(J;X_1^s)}\le M_{7}\big(|y_0|_{D_{A_s}(\alpha,\infty)}+
\no f\no _{B^\alpha_{1\infty}(J;X^s_0)}\big).
$$
Note that this estimate is in particular independent of $J=(0,a)$, by exponential stability of $e^{-A_st}$.
Further we have
$
y_0\in X_\gamma\cap X^s= D_{A_s}(1-1/p,p)\hookrightarrow D_{A_s}(\alpha, \infty),
$
provided $\alpha\le 1-1/p$. 
Another parabolic estimate valid for \eqref{SCP} that
we shall make use of reads as
$$
\no y\no _{B^\alpha_{1\infty}(J;X^s_1)}\le
 M_{8}\big(|y_0|_{D_{A_s}(\alpha,\infty)}+\no f\no _{L_1(J_1;X^s_1)}),
\quad 
$$
provided $\alpha<1$. Here the constant $M_{8}$ is also independent of 
$J=(0,a)$.
\smallskip\\
We set $R_1(t)=-\phi_s^\prime(x(t))T(x(t),y(t),z(t))$ and recall that 
$|\phi_s^\prime(x(t))|_{\cB(X^c,X_1)}\le  1$
for $t\in [0,t_1]$. Employing the $L_1$-estimate for $z$,
see \eqref{z-Lq}, yields
\begin{align*}
\no R_1\no _{L_1(J_1;X_1)}&\le  \int_0^{t_1}|T(x(s),y(s),z(s))|\,ds
\le  \beta (\no y\no _{L_1(J_1;X_1)}+\no z\no _{L_1(J_1;X_0)})\\
&\le  \beta M_6\rho +\beta (1+\beta M_6)\no y\no _{L_1(J_1;X_1)}.
\end{align*}  
 Therefore, for the solution $y_1$ of
\eqref{SCP} with $f=R_1$ we obtain
$$\no y_1\no _{B^\alpha_{1\infty}(J_1;X_1)}
\le
 M_8\big(|y_0|_{\gamma}
 +\beta M_6\rho +\beta(1+M_6\beta)\no y\no _{L_1(J_1;X_1)}\big).
 $$
Next let $R_2(t)=P^s(G(\Phi(x)+y+z)-G(\Phi(x)))$. 
Then by estimate \eqref{G-estimate}
\begin{align*}
\no R_2\no _{L_1(J_1;X_0)}
&\le  \beta (\no y\no _{L_1(J_1;X_1)}
+\no z\no _{L_1(J_1;X_0)})\\
&\le   \beta M_6\rho +\beta(1+M_6\beta)\no y\no _{L_1(J_1;X_1)},
\end{align*}  
and with some constant $C_6$ 
\begin{align*}
|R_2(t)-R_2(\bar{t})|_0
&\le  C_6\beta\big(|y(t)-y(\bar{t})|_1 +|z(t)-z(\bar{t})|
+ |x(t)-x(\bar{t})|\big)\\
&+ C_6|y(t)|_1\big(|y(t)-y(\bar{t})|_\gamma+|x(t)-x(\bar{t})|+|z(t)-z(\bar{t})|\big).
\end{align*}
Hence we obtain the following estimate 
\begin{align*}
[R_2]_{\alpha,0}
&\le  C_6\beta \big\{[y]_{\alpha,1} + 
[z]_{\alpha,0}+[x]_{\alpha,0}\big\}
+C_6\sup_{0< h\le h_1}\!\!h^{-\alpha}\int_0^{t_1-h}|y(t)|_1\\
&\cdot\big\{|y(t+h)-y(t)|_\gamma
+|x(t+h)-x(t)|+|z(t+h)-z(t)|\big\}\,dt
\end{align*}
where we set $h_1:=\min(1,t_1)$ and $[\,\cdot\,]_{\alpha,j}:=[\,\cdot\, ]_{J_1;\alpha,X_j}$
for $j=0,1.$
\eqref{z-Lq}--\eqref{HypEst-z} yields for each $\alpha\in(0,1)$
$$[z]_{\alpha,0}\le  \no \dot{z}\no _{L_1(J_1;X_0)}\le 
 C_7(\rho+\beta\no y\no _{L_1(J_1;X_1)}),$$
with some uniform constant $C_7$.
In the same way we may estimate $[x]_{\alpha,0}$. 
Next we have again by \eqref{z-Lq}--\eqref{HypEst-z} 
\begin{align*}
\sup_{0<h\le h_1} h^{-\alpha}\int_0^{t_1-h}|y(t)|_1|z(t+h)-z(t)|\,dt
&\le  h^{1/p^\prime-\alpha}\no \dot{z}\no _{L_p(J_1;X_0)}
\no y\no _{L_1(J_1;X_1)}\\
&\le  C_8(|y_0|_\gamma+\rho)\no y\no _{L_1(J_1;X_1)}
\end{align*}
provided $\alpha\le  1-1/p$, and similarly for the corresponding integral 
containing the $x$-difference.
Last but not least, 
for $\alpha\le  (1-\gamma)(1-1/p)$ we have by 
\eqref{HypEst-y}
\begin{align*}
\sup_{0< h\le h_1}h^{-\alpha}
&\int_0^{t_1-h}|y(t)|_1|y(t+h)-y(t)|_\gamma\,dt\\
&\le 
2 C h^{(1-\gamma)/p^\prime-\alpha}\no \dot{y}\no ^{1-\gamma}_{L_p(J_1;X_0)}
\no y\no _{L_{p}(J_1;X_1)}\no y\no _{L_{1}(J_1;X_1)}^{\gamma}\\
&\le  C_9(|y_0|_\gamma+\beta\rho)^{2-\gamma}\no y\no _{L_{1}(J_1;X_1)}^{\gamma}\\
&\le C_{10}\big((|y_0|_\gamma+\beta\rho)^2
+ (|y_0|_\gamma+\beta\rho)\no y\no_{L_1(J_1;X_1)}
\big),
\end{align*}
where we used Young's inequality in the last line. \\

Collecting now all terms and choosing $\alpha=(1-\gamma)/p^\prime = 1/pp^\prime$,
we  find a uniform constant $C_{11}$ such that for $|y_0|_\gamma\leq\delta$
$$\no y\no _{B^\alpha_{1\infty}(J_1;X_1^s)}
\le C_{11}\big(|y_0|_\gamma+\beta\rho +
(\beta+\rho+\delta)\no y\no _{B^\alpha_{1\infty}(J_1;X_1^s)}\big),$$
hence 
\begin{equation}
\label{y-L1}
\no y\no _{L_1(J_1;X_1^s)}\le \no y\no _{B^\alpha_{1\infty}(J_1;X_1^s)}
\le 2C_{11}(|y_0|_\gamma+ \beta \rho),
\end{equation}
provided $ C_{11}(\beta+\rho+\delta) <1/2$.
Choosing now first $\beta$, i.e.\ $\eta$ and $r$  
small enough, and then $\rho$ and 
$\delta>0$, we see 
that $|u(t_1)-u_*|_\gamma <3\rho$, a contradiction to $t_1<t_*$. 
As in (e) of the proof
of Theorem 2.1 we may then conclude that $t_*=\infty$, 
which means that the solution 
exists globally and stays in the ball 
$\bar{B}_{X_\gamma}(u_*,3\rho)$. 
\medskip\\
(f) To prove convergence, let $(x(t),y(t),z(t))$ be a global solution of
\eqref{hypcanform} that satisfies
\[
 |x(t)|, |y(t)|_\gamma,  |z(t)|\le 3\rho,\quad \mbox{for all}\; t\ge 0, \quad
\]
see \eqref{XYZ}.
Similarly to the proof of Theorem~\ref{th:1}, part (e), we obtain
from \eqref{estimate-Rl-T}
\begin{equation}
\label{y-estimate-II} \no e^{\omega t} y\no_{\EE_1(\infty)} \le
2M_1|y_0|_\gamma + 2\beta M_0 \no e^{\omega t}z\no_{\EE_0(\infty)},
\end{equation}
where $\omega\in (0,\inf \{\mbox{Re}\,\lambda: \lambda\in
\sigma(A_s)\})$ is a fixed number and $\beta$ is given in
\eqref{beta}. Repeating the estimates leading up to
\eqref{y-estimate-gamma} we get
\begin{equation}
\label{y-estimate-II-gamma} |e^{\omega t}y(t)|_\gamma \le M_2
|y_0|_\gamma + 2\beta c_0M_0 \no e^{\omega t}z\no_{\EE_0(\infty)}, \quad
t\ge 0.
\end{equation}
From equation \eqref{z-backwards} we infer that
\begin{equation}
\label{zident}
z(t)= -\int_t^\infty e^{-A_u(t-s)}
     R_u(x(s),y(s),z(s))\,ds,\quad t\ge 0,
\quad 
\end{equation}
since $|z(t_1)|\le\rho$ for each $t_1>0$ and $e^{A_u(t_1-t)}$ 
is exponentially decaying for $t_1\to \infty$.
Using (\ref{zident})
and the estimate for $R_u$ from (\ref{estimate-Rl-T}) and proceeding
as in the proof of Young's inequality for convolution integrals one
shows that
\begin{equation} \label{zweight}
\no e^{\omega t} z\no_{\EE_0(\infty)}\le C_{12}\beta\big(\no e^{\omega
t} y\no_{\EE_1(\infty)}+\no e^{\omega t} z\no_{\EE_0(\infty)}\big).
\end{equation}
Making $\beta$ sufficiently small (by decreasing $\eta$ and,
accordingly, $r$) it follows from (\ref{y-estimate-II}) and
(\ref{zweight}) that
\[
\no e^{\omega t} y\no_{\EE_1(\infty)}+\no e^{\omega t}
z\no_{\EE_0(\infty)}\le C_{13}|y_0|_\gamma.
\]
This estimate in turn, together with (\ref{y-estimate-II-gamma}),
implies $|y(t)|_\gamma\to 0$ and $|z(t)|\to 0$ exponentially fast as
$t\to\infty$. As in the proof of Theorem~\ref{th:1} part (f) we get
$$x(t)\to x_\infty:=x_0+\int_0^\infty T(x(s),y(s),z(s))\, ds.$$
This yields existence of the limit
$$u_\infty=u_\ast+v_\infty :=u_\ast+\lim_{t\to\infty} v(t)
= u_\ast + x_\infty +\phi_s(x_\infty) +\phi_u(x_\infty)\in\cE.$$
Similar arguments as in Section 2 yield exponential convergence of
$u(t)$ to $u_\infty$. 
\end{proof}

\noindent A  result similar to Theorem~\ref{th:3} is also valid in
the setting of Section 3. We leave the details to the interested
reader.
\bigskip

\bibliographystyle{plain}

\end{document}